\documentclass[a4paper,11pt]{article}
\usepackage{amssymb, amsmath, amsthm, graphicx}

\newtheorem{thm}{Theorem}[section]
\newtheorem{dfn}{Definition}[section]
\newtheorem{lem}[thm]{Lemma}
\newtheorem{pro}[thm]{Proposition}
\newtheorem{cor}[thm]{Corollary}

  \makeatletter
  \@addtoreset{equation}{section}
  \makeatother
\usepackage[top=40truemm,bottom=40truemm,left=30truemm,right=30truemm]{geometry}

\title{\bfseries{\Large{Generation of interfaces for multi-dimensional stochastic Allen-Cahn equation with a noise smooth in space}}}
\author{\Large{Kai Lee} \\ $University \  of \  Tokyo$}
\date{}

\begin{document}

\maketitle

\begin{abstract}
In this paper, we study the generation of interfaces for a stochastic Allen-Cahn equation with general initial value \cite{kl} in the multi-dimensional case that external noise is given by $Q$-Brownian motion. We prove that interfaces, for $d$-dimensional stochastic Allen-Cahn equation with scaling parameter $\varepsilon >0$, are generated at the time of order $O(\varepsilon |\log \varepsilon|)$. Especially, in one-dimensional case, we give more detailed estimate and shape of the generated interface than that obtained in \cite{kl}. Assuming that the $Q$-Brownian motion is smooth in space variable, we extend a comparison theorem for PDE to SPDE's in order to prove the generation. Moreover, we connect the generated interface to the motion of interface in one-dimension \cite{f94}. In this case, we consider the white noise only in time multiplied by $\varepsilon ^{\gamma} a(x)$ as the noise term, where $a$ is a smooth function which has a compact support. This is the special case of $Q$-Brownian motion. We take the time scale of order $O(\varepsilon^{-2\gamma - 1})$ for studying the motion of interface. 
\end{abstract}

\section{Introduction}
\label{sec1}
In this paper, we study the sharp interface limit for multi-dimensional stochastic Allen-Cahn equation with Neumann boundary condition:
\begin{align}
\label{eq:spde2}
\begin{cases}
\dot{u} ^\varepsilon (t,x) = \Delta u^\varepsilon (t,x)+\displaystyle{\frac{1}{\varepsilon}} f(u^\varepsilon (t,x) ) + \dot{W} _t ^\varepsilon (x),\ \ \ t> 0,\  x\in D,\\
u^\varepsilon (0,x) = u_0 (x),\ \ \ x\in D, \\
\frac{\partial u}{\partial \nu} (t,x) = 0,\ \ \ t> 0,\  x\in \partial D
\end{cases}
\end{align}
where $D$ is a domain of $\mathbb{R}^d$, for $d\geq 2$, which has a $C^1$ boundary and $\nu$ is the exterior unit normal vector of $\partial D$. We use the notation of $\dot{u} = \frac{\partial u}{\partial t}$, $\Delta u=\sum _{i=1} ^d \frac{\partial ^2 u}{\partial x_i ^2}$ and $\dot{W} _t ^\varepsilon (x)$ is an external random noise. We give a mathematical meaning to the solution $u^\varepsilon$ as a mild solution or a generalized solution (see \cite{dpz}). We assume that the bistable reaction term $f$ has $\pm 1$ as stable points and satisfies $\int _{-1} ^1 f(u)du=0$.

In \cite{kl}, we considered one-dimensional case and computed a generation time of the solution $u^\varepsilon$ in the case that $\dot{W} _t ^\varepsilon (x):=\varepsilon ^\gamma a(x) \dot{W} _t (x)$. We assumed that $\dot{W} _t (x)$ is a space time white noise and $a\in C_0 ^\infty (\mathbb{R})$. We obtained that the first generation time is of order $O(\varepsilon |\log \varepsilon |)$ and connected to the motion of interface which is the result of Funaki \cite{f}. The proper time scale for the motion is of order $O(\varepsilon ^{-2\gamma -\frac{1}{2}})$.

In this paper, we show the generation of interface in $d$-dimensional setting for $d \geq 1$ in the case that the external noise is smooth in space. Especially, in the case of $d=1$, we get more detailed shape of the generated interface than the result of \cite{kl}. Moreover, we consider the motion of interface in the case of $\dot{W} _t ^\varepsilon (x):=\varepsilon ^\gamma a(x) \dot{W} _t$ for $\varepsilon >0$ where the function $a$ is same as above and $\dot{W} _t$ is a white noise only in time. We can regard the noise $a(x) \dot{W} _t$ as a special case of $Q$-Brownian motion. Funaki \cite{f94} investigated the dynamics of the interface in this case with an initial value which has already formed an interface and show that the proper time scale is of order $O(\varepsilon ^{-2\gamma -1})$. We consider more general initial value, compute the generation time and connect to the result of \cite{f94}.

\subsection{Setting of the model}
At first, we consider $d$-dimensional equation (\ref{eq:spde2}) for $d \geq 2$ with Nenmann boundary condition. The external noise term is defined by $\dot{W} _t ^\varepsilon (x):=\varepsilon ^\gamma \dot{W} _t ^{Q_d} (x)$ where $\varepsilon >0$. The noise $\dot{W} _t ^{Q_d} (x)$ is a formal time derivative of a $Q_d$-Brownian motion on $\mathbb{R} ^d$, which is smooth in a space variable and has a covariance structure;
\begin{align}
\label{cov}
E[W^{Q_d} _t (x) W^{Q_d} _s (y)] = (t\wedge s) Q_d (x,y),
\end{align}
where the function $Q_d :D \times D \to \mathbb{R}$ is a positive, symmetric and compactly supported smooth function. We assume that $\frac{\partial}{\partial x} \frac{\partial}{\partial y} Q_d (x,y)$ and $\frac{\partial ^2}{\partial x^2} \frac{\partial ^2}{\partial y^2} Q_d (x,y)$ are positive if $x=y$. The initial value $u_0 \in C^2(\overline{D})$ satisfies
\begin{align}
\label{eq:ini2}
\| u_0 \| _\infty + \| u_0 '  \| _\infty + \| u_0 '' \| _\infty \leq C_0,
\end{align}
where $\| \cdot \| _\infty$ is the supremum norm of $C(D)$.

Next we consider 1-dimensional case;
\begin{align}
\label{eq:spde}
\begin{cases}
\dot{u} ^\varepsilon (t,x) = \Delta u^\varepsilon (t,x)+\displaystyle{\frac{1}{\varepsilon}} f(u^\varepsilon (t,x) ) + \dot{W} _t ^{\varepsilon} (x),\ \ \ t> 0,\  x\in \mathbb{R} ,\\
u^\varepsilon (0,x) = u_0 ^\varepsilon (x),\ \ \ x\in \mathbb{R} ,\\
u^\varepsilon (t,\pm \infty)= \pm 1,\ \ \ t\geq  0.
\end{cases}
\end{align}
We define the external noise by $\dot{W} _t ^\varepsilon (x):=\varepsilon ^\gamma \dot{W} _t ^{Q} (x)$ where $W _t ^{Q} (x)$ is a $Q$-Brownian motion on $\mathbb{R}$ which has the same covariance as (\ref{cov}) where $Q :\mathbb{R} \times \mathbb{R} \to \mathbb{R}$ is a positive, symmetric and compactly supported smooth function. We assume that $\frac{\partial}{\partial x} \frac{\partial}{\partial y} Q (x,y)$ and $\frac{\partial ^2}{\partial x^2} \frac{\partial ^2}{\partial y^2} Q (x,y)$ are positive if $x=y$. We also assume that $u_0 ^\varepsilon \in C^2(\mathbb{R})$ and there exist constants $C_0 >1$, $C$, $C'$, $\kappa$, $K>0$ and a function $g_1$, $g_2 \in H^3(\mathbb{R})$ such that
\begin{align}
\label{eq:ini}
 \begin{cases}
   \text{(i)} \| u_0 ^\varepsilon \| _\infty + \| u_0 ^{\varepsilon \prime}  \| _\infty + \| u_0 ^{\varepsilon \prime \prime} \| _\infty \leq C_0, & \\
   \text{(ii)} \text{There exists a unique zero } \xi _0 \in [-K,K] \text{ such that } u_0 ^\varepsilon (\xi _0)= 0, & \\
   \text{(iii)} | u_0 ^\varepsilon (x)| \geq C \varepsilon ^{\frac{1}{2}}\ (|x-\xi _0| \geq C' \varepsilon ^{\frac{1}{2}}), & \\
   \text{(iv)} | u_0 ^\varepsilon (x) - 1 | \leq \varepsilon ^{\kappa } g_1(x) \  (x\geq K), & \\
   \text{(v)} | u_0 ^\varepsilon (x) + 1 | \leq \varepsilon ^{\kappa } g_2(x) \  (x\leq  -K), & \\
 \end{cases}
\end{align}
where $\| \cdot \| _\infty$ is the supreme norm of $C(\mathbb{R})$. In this paper, we take $K=1$ and assume that the support of $Q(x,y)$ is included in $[-1,1]\times [-1,1]$ without loss of generality. 

The reaction term $f\in C^2( \mathbb{R} )$ of SPDEs (\ref{eq:spde2}) and (\ref{eq:spde}) satisfies the following conditions:
\begin{align}
\label{reaction}
 \begin{cases}
  \text{(i)}f \text{ has only three zeros }\pm 1, 0, & \\
  \text{(ii)}f'( \pm 1) =-p < 0,\  f'(0)=\mu > 0, & \\
  \text{(iii)}f(u)\leq C(1+|u|^q)\text{ with some }C,q>0,& \\
  \text{(iv)}f'(u)\leq c\text{ with some }c>0, & \\
  \text{(v)}f\text{ is odd,} &\\
  \text{(vi)}f(u)\leq -p(u-1) \ (u\geq 1).
 \end{cases}
\end{align}
The reaction term is bistable and has only $u=\pm 1$ as stable points from (i) and (ii). The existence of global solutions for the SPDEs (\ref{eq:spde2}) and (\ref{eq:spde}) with noises introduced above is insured by (iii). We need the assumption (iv) in order to use a comparison theorem by applying the maximal principle for the parabolic PDEs (See Section 2 of \cite{fr}). The condition (v) implies $\int _{-1}^{1}f(u)du=0$. We assume (vi) for a technical reason. For example, we can take $f(u)=u-u^3$. Throughout this paper, we set $C_f := \sup _{u \in [-2C_0,2C_0]} f'(u)$. 


\subsection{Main results}
At first, we show the generation of interface in $d$-dimensional case for $d\geq 2$. We get a following result for the generation of interface.
\begin{thm}
\label{thm1}
Let $u^\varepsilon$ be the solution of (\ref{eq:spde2}) which satisfy (\ref{reaction}) and (\ref{eq:ini2}). If there exist constants $C_1 >0$, $\kappa$ and $\alpha$ satisfying $\kappa > \alpha >\frac{1}{2}$, $\kappa >1$ and $\frac{\alpha}{\mu} + \frac{\kappa}{p} < C_1 < \frac{1}{\mu}$, then there exist positive constants $\tilde{\gamma} _d >0$ and, for all $\gamma \geq\tilde{\gamma} _d$, we have that
\begin{align}
&{\rm (i)}\lim _{\varepsilon \to 0}P(-1-\varepsilon ^\kappa \leq u^\varepsilon (x, C_1 \varepsilon | \log \varepsilon |) \leq 1+\varepsilon ^\kappa \ for \ all \ x\in D) =1\\
&{\rm (ii)}\lim _{\varepsilon \to 0}P(u^\varepsilon (x, C_1 \varepsilon | \log \varepsilon |) \geq 1-\varepsilon ^\kappa \ for \ x\in D \ such \ that \ u_0(x) \geq \varepsilon ^{1-C_1 \mu}) =1\\
&{\rm (iii)}\lim _{\varepsilon \to 0}P(u^\varepsilon (x, C_1 \varepsilon | \log \varepsilon |) \leq -1+\varepsilon ^\kappa \ for \ x\in D \ such \ that \ u_0(x) \leq -\varepsilon ^{1-C_1 \mu}) =1
\end{align}
\end{thm}

Next, we consider one-dimensional equation (\ref{eq:spde}). Before we state theorems, we mathematically formulate the generation of interface in one-dimension.
\begin{dfn}
\label{def21}
We say $u^\varepsilon (t,\cdot)$ generates an interface if there exist $\beta ,\  \kappa >0,\ C>0,\ K>0$ and $\bar{g}_1$, $\bar{g}_2(x) \in H^1(\mathbb{R})$ such that
 \begin{align}
 \label{gen}
  \begin{cases}
   \text{\rm{(i)} }| u ^\varepsilon (t,x) | \leq 1+\varepsilon ^\kappa \ (x \in [-1,1]), & \\
   \text{\rm{(ii)} }u ^\varepsilon (t,x) \geq 1-\varepsilon ^\kappa \  (x \in [-1,1] \text{ such that } u_0(x) \geq C \varepsilon ^{\beta} ), & \\
   \text{\rm{(iii)} }u ^\varepsilon (t,x) \geq -1+\varepsilon ^\kappa \  (x \in [-1,1] \text{ such that } u_0(x) \leq -C \varepsilon ^{\beta} ), & \\
   \text{\rm{(iv)} }| u ^\varepsilon (t,x) - 1 | \leq \varepsilon ^\kappa \bar{g}_1(x) \  (x\geq 1 ), & \\
   \text{\rm{(v)} }| u ^\varepsilon (t,x) + 1 | \leq \varepsilon ^\kappa \bar{g}_2(x) \  (x\leq  -1). & \\
 \end{cases}
\end{align}
\end{dfn}
Now we state generation and motion of interface as our main results.
\begin{thm}
\label{thm21}
Let $u^\varepsilon$ be the solution of (\ref{eq:spde}) which satisfies (\ref{reaction}) and (\ref{eq:ini}). If $\dot{W}^\varepsilon (t,x) := \varepsilon ^\gamma \dot{W}^Q (t,x)$ and there exist constants $C_1 >0$, $\kappa$ and $\alpha$ satisfying the same condition as in Theorem \ref{thm1}, then there exists a constant $\tilde{\gamma} >0$ and, for all $\gamma \geq\tilde{\gamma}$, we have that
\begin{align}
\lim _{\varepsilon \to 0} P(u^\varepsilon ( C_1 \varepsilon | \log \varepsilon | ,\cdot)\text{ generates an interface}) = 1,
\end{align}
and $\beta$ in (\ref{gen}) is taken as $\beta =1-C_1 \mu$.
\end{thm}
Next we formulate the result of the motion of interface for a special $Q$-Brownian motion. Indeed we take $a(x) \dot{W} _t$ as $\dot{W} _t ^Q (x)$ in order to connect Theorem \ref{thm21} to  Theorem1.1 of \cite{f94}, p.135. The initial value in the following theorem is more general than \cite{f94}.
\begin{thm}
\label{thm22}
Let $u^\varepsilon$ be the solution of (\ref{eq:spde}) which satisfies (\ref{reaction}) and (\ref{eq:ini2}) and set $\bar{u}^\varepsilon (t,x) := u^\varepsilon (\varepsilon ^{-2\gamma -1} t,x)$. If $\dot{W}^\varepsilon (t,x) := \varepsilon ^\gamma a(x) \dot{W} _t$ and there exist constants $C_1 >0$, $\kappa$ and $\alpha$ satisfying the same condition as in Theorem \ref{thm1}, then there exists a constant $\tilde{\gamma} >0$ and, for all $\gamma \geq \tilde{\gamma}$, we have that
\begin{align}
P( \| \bar{u} ^\varepsilon (t,\cdot )-\chi _{\xi^\varepsilon _t}(\cdot )\| _{L^2(\mathbb{R})} \leq \delta \  for \ all \  t \in [C_1 \varepsilon | \log \varepsilon |,T]  ) \to 1\ \ \ (\varepsilon \to 0).
\end{align}
Here the distribution of the process $\xi^\varepsilon _t$ on $C([0,T], \mathbb{R})$ converges to that of $\xi _t$ weakly and $\xi_t$ obeys the SDE starting at $\xi_0$ (see (\ref{eq:ini}) for $\xi_0$);
\begin{align}
 d\xi_t = \alpha _1 a(\xi_t ) dB_t + \alpha _2 a(\xi_t)a'(\xi_t)dt,
\end{align}
with certain $\alpha _1$ and $\alpha _2 \in \mathbb{R}$, see \cite{f94}, p. 135.
\end{thm}

These two results imply that the interface is formed by the early time of order $O(\varepsilon |\log \varepsilon |)$ and afterward move in the time scale of order $O(\varepsilon ^{-2\gamma -1})$. This time scale is same as that of Funaki's result \cite{f94} in the case that the external noise is $\varepsilon ^\gamma a(x) W_t$ (\cite{f94} discussed not only $\varepsilon ^\gamma a(x) W_t$ but also $\varepsilon ^\gamma a(x) W_t ^h$ where $W_t ^h$ is a $Q$-Brownian motion which has a Riesz potential kernel as a covariance operator).

The proofs of these results are based on the methods of Alfaro et al \cite{ham}. They computed the first generation time for the multi-dimensional case with non-random deterministic external force. Their main idea was to construct super and sub solutions of the Allen-Cahn equation;
\begin{align}
\label{eq:pde}
\begin{cases}
\dot{u} ^\varepsilon (t,x) &= \Delta u ^\varepsilon (t,x)+\displaystyle{\frac{1}{\varepsilon}} f(u ^\varepsilon (t,x) ) + g^\varepsilon (t,x),\ \ \ t> 0,\  x\in D \subset \mathbb{R} ^d,\\
u^\varepsilon (0,x) &= u_0 (x),\ \ \ x\in D,\\
\frac{\partial u}{\partial \nu} (t,x)&= 0,\ \ \ t> 0,\  x\in \partial D,
\end{cases}
\end{align}
where $D$ is a domain with a smooth boundary. For simplicity, we assume that $g^\varepsilon \equiv 0$ and $u_0$ satisfies Neumann boundary condition. In a very short time, the effect of the diffusion term, that is $\Delta u^\varepsilon$, is negligible compared with the other. Thus they considered an ODE:
\begin{align}
\begin{cases}
\label{eq:ode}
\dot{Y} (\tau , \xi) = f( Y(\tau , \xi )),\ \ \  \tau >0,\\
Y(0, \xi)=\xi \in [-2C_0 ,2C_0],
\end{cases}
\end{align}
(see (\ref{eq:ini}) for the constant $C_0$), took a positive constant $C_2 >0$ and defined
\begin{align}
w_\varepsilon ^\pm (t,x) = Y \left( \frac{t}{\varepsilon} , u_0 (x) \pm \varepsilon  C_2 (e^\frac{\mu t}{\varepsilon }-1 ) \right)
\end{align}
where $\mu$ is defined in (\ref{reaction}), and proved that $w_\varepsilon ^- \leq u^\varepsilon \leq w_\varepsilon ^+$ by applying a comparison theorem for parabolic PDEs until the time of order $O(\varepsilon |\log \varepsilon |)$. And they proved that $w_\varepsilon ^\pm$ formed the interfaces. If $g^\varepsilon$ is more general or $u_0$ does not satisfy Neumann boundary condition, we need to modify these arguments slightly. After the generation of interfaces, the diffusion term becomes much larger, and balances with the reaction term. Thus they took another super and sub solutions and connect to the interface motion which called the motion by mean curvature with the effect of non-random external force.

In this paper, we construct super and sub solutions, derive the first generation time and prove the motion of interface in the case that the noise is $Q$-Brownian motion. In particular, we show after the time $\frac{1}{2\mu} \varepsilon |\log \varepsilon |$, which is generation time in \cite{ham}, that a long time scale of order $O(\varepsilon ^{-2\gamma-1})$ is proper because of the effect of noise. And we connect to the results of \cite{f} by using the strong Markov property. Moreover, in \cite{ham}, they constructed super and sub solutions applying directly the maximal principle for PDE, however, this is not straightforward because the solution is singular in a time variable. And thus, we prove the comparison theorem of SPDE by approximating solutions smoothly instead. These are the major difference from the result of \cite{ham}.

We first prepare some estimates in Section \ref{sec3}. In Section \ref{sec4}, we find the SDE which corresponds to SPDEs (\ref{eq:spde2}) and (\ref{eq:spde}), and prove that the solution of this SDE is close to that of ODE (\ref{eq:ode}) in order to use the results of the PDE case. We extend the comparison theorem of PDE to SPDE. For this purpose, we smoothly and uniformly approximate the solution of SPDE in Section \ref{sec5}. In Section \ref{sec6}, we apply the comparison theorem to the approximated solution which is constructed in Section \ref{sec5}, and we construct super and sub solutions of $u^\varepsilon$. At last, we give the proof of Theorem \ref{thm1} in Section \ref{sec8} and the proofs of Theorem \ref{thm21} and Theorem \ref{thm22} in Section \ref{sec7}.

\section{Auxiliary results}
\label{sec3}
From this section, we consider the case that the external noise $\dot{W} _t ^\varepsilon (x)$ is $\varepsilon ^\gamma \dot{W} _t ^{Q_d} (x)$ or $\varepsilon ^\gamma \dot{W} _t ^{Q} (x)$ which are defined in Section \ref{sec1}. 

At first, we refer to the conditions of the solutions of (\ref{eq:spde2}) and (\ref{eq:spde}) $u^\varepsilon$; see Section 2 of \cite{f} or Theorem 3.1 of \cite{f94}.

\begin{pro}
\label{thm31}
If $|u^\varepsilon _0 (x)|\leq K$ and $u^\varepsilon$ is the solution of SPDE (\ref{eq:spde2}) (resp. (\ref{eq:spde})), then
\begin{align}
\lim_{\varepsilon \to 0} P\left ( |u^\varepsilon  (t,x)| \leq \max \{K,1\} +\delta,\ t \in [0,\varepsilon ^{-n}],\ x\in D\ (resp. \ x\in \mathbb{R}) \right ) =1,
\end{align}
for all $n \in \mathbb{N}$ and $\delta >0$.
\end{pro}

From this result, we see that the solution $u^\varepsilon$ takes value in the interval $[-2C_0 ,2C_0]$ with high probability. By introducing stopping times
\begin{align}
\label{bddst}
&\tau _1 :=\inf \{t >0 |  |u^\varepsilon  (t,x)| > 2C_0\ for\ some\ x\in D \},\\
&\bar{\tau} _1 :=\inf \{t >0 |  |u^\varepsilon  (t,x)| > 2C_0\ for\ some\ x\in \mathbb{R} \},
\end{align}
we can characterize Proposition \ref{thm31} as $P(\tau _1 >\varepsilon ^{-n}) \to 1$ and $P(\bar{\tau} _1 >\varepsilon ^{-n}) \to 1$ ($\varepsilon \to 0$).

Next we refer to some preliminary results of ODE (\ref{eq:ode}):
\begin{align}
\begin{cases}
\dot{Y} (\tau , \xi) = f( Y(\tau , \xi )),\ \ \  \tau >0,\nonumber \\
Y(0, \xi)=\xi \in [-2C_0 ,2C_0]. \nonumber
\end{cases}
\end{align}
See \cite{kl} for the detailed proofs.

\begin{pro}
\label{thm32}
For any $\alpha >0$ and $\kappa>0$, there exists a positive constant $C_1> \frac{\alpha}{\mu} + \frac{\kappa}{p}$ such that
\begin{align}
&|Y(C_1|\log \varepsilon | ,\xi )-1 |\leq \varepsilon ^\kappa \ \ \ for\ all\ \xi \in [\varepsilon ^\alpha ,2C_0]\\
&|Y(C_1|\log \varepsilon | ,\xi ) +1 |\leq \varepsilon ^\kappa \ \ \ for\ all\ \xi \in [-2C_0 ,-\varepsilon ^\alpha ]
\end{align}
for sufficiently small $\varepsilon >0$ where $\mu$ and $p$ are defined in (\ref{reaction}). The constant $C_1$ can be taken depending only on $\alpha$, $\kappa$ and $f$.
\end{pro}

Here we derive the estimate of two solutions of ODE (\ref{eq:ode}) as a corollary of Lemma 2.1 of \cite{kl}. We use this estimate in Section \ref{sec4}.

\begin{cor}
\label{cor31}
For all $\alpha >0$, $\beta >0$ and $\delta \geq \frac{2\alpha C_f}{\mu}$, there exists a positive constant $C>0$ such that\\
{\rm (i)} if $\xi \in [\varepsilon ^\alpha ,2C_0-\varepsilon ^{\beta +\delta}]$  then for sufficiently small $\varepsilon >0$,
\begin{align}
Y(\tau ,\xi+\varepsilon ^{\beta +\delta}) - Y(\tau ,\xi)  \leq \varepsilon ^\beta \ \ \ for\ all\ \tau \in [0,C|\log \varepsilon |].
\end{align}
{\rm (ii)} if $\xi \in [-2C_0+\varepsilon ^{\beta +\delta} , - \varepsilon ^\alpha]$  then for sufficiently small $\varepsilon >0$,
\begin{align}
Y(\tau ,\xi) - Y(\tau ,\xi - \varepsilon ^{\beta +\delta})  \leq \varepsilon ^\beta \ \ \ for\ all\ \tau \in [0,C|\log \varepsilon |].
\end{align}
\end{cor}
\begin{proof}
We only prove (i). At first we take $\eta >0$ small enough so that $f'(u)<0$ holds if $u \in [1-\eta , 1+\eta]$. From Lemma 2.1 of \cite{kl}, $Y(\tau ,\xi)$ is in the interval $[1-\eta , 1+\eta]$ by the time $T_\varepsilon := \frac{2\alpha}{\mu} |\log \varepsilon |$. After the time $T_\varepsilon$, $\dot{Y}(\tau ,\xi+\varepsilon ^{\beta +\delta}) - \dot{Y}(\tau ,\xi )=f(Y(\tau ,\xi+\varepsilon ^{\beta +\delta})) - f(Y(\tau ,\xi ))$ keeps the sign negative. Thus $Y(\tau ,\xi+\varepsilon ^{\beta +\delta}) - Y(\tau ,\xi)$ declines but does not touch 0 after $T_\varepsilon$. We need to consider the behavior of $Y(\tau ,\xi+\varepsilon ^{\beta +\delta}) - Y(\tau ,\xi)$ before the time $T_\varepsilon$. In the case of $\xi \in [\varepsilon ^\alpha , 1-\eta]$, we obtain 
\begin{align}
\label{est3-2}
Y(\tau ,\xi+\varepsilon ^{\beta +\delta}) - Y(\tau ,\xi) \leq \varepsilon ^{\beta +\delta}\exp (C_f \tau) \leq \varepsilon ^{\beta +\delta - \frac{2\alpha C_f}{\mu}} \leq \varepsilon ^\beta \ \ \ for\  all\  \tau \in [0,T_\varepsilon ]
\end{align}
from Gronwall's inequality. It is easy to show (\ref{est3-2}) in the case of $\xi \in [1+\eta ,2C_0]$ because the time at which $Y(\tau ,2C_0)$ is in the interval $[1-\eta , 1+\eta]$ is independent of $\varepsilon$. And thus, we obtain the same estimate as (\ref{est3-2}).
\end{proof}

\section{Estimates for SDE}
\label{sec4}
In this section, we prove some estimates for solutions of SDE, the estimates which we often use in this paper. We change the variable $t$ to $\varepsilon \tau$. In order to construct the solutions, we consider the solution of SDE;
\begin{align}
\label{sde}
\begin{cases}
\dot{Y} ^\varepsilon (\tau, \xi ,x) = f( Y ^\varepsilon (\tau, \xi ,x )) + \varepsilon ^{\gamma + \frac{1}{2}} \dot{\widetilde{W}}_\tau ^{Q_d} (x),\ \ \  \tau \in (0,\infty ),\\
Y ^\varepsilon (0, \xi ,x)=\xi \in [-2C_0 ,2C_0],
\end{cases}
\end{align}
where $\widetilde{W}_\tau ^{Q_d} (x) := \varepsilon ^{-\frac{1}{2}} W_{\varepsilon \tau} ^{Q_d} (x)$ is a $Q_d$-Brownian motion in law sense. Moreover this process equals to $\sqrt{Q_d (x,x)} W_\tau$ in law sense for each $x\in D$. For simplicity, we denote $\widetilde{W}_\tau ^{Q_d} (x)$ as $W_\tau (x)$. In this section, we prove that $Y^\varepsilon$ stays close to $Y$ which is the unique solution of the ODE (\ref{eq:ode}) for a long time with high probability, and construct the super and sub solutions of the SPDE (\ref{eq:spde2}) using $Y^\varepsilon$. Now let $\delta >0$ be the small positive constant such that $f'(x)<0 $ for all $x \in [1-\delta,1+\delta]$, and let $\tau _\varepsilon$ be a stopping time defined by
\begin{align}
\tau _\varepsilon ^p:= \inf \{ \tau >0 | \| Y^\varepsilon (\tau \wedge \tau_{\varepsilon} , \xi ,\cdot) - Y (\tau \wedge \tau_{\varepsilon} , \xi)  \|_{W^{1,2p}(D)} >\varepsilon ^{\kappa} \},
\end{align}
for each $\xi \in [-2C_0,2C_0]$. And we define a deterministic time as $T_\varepsilon := \frac{1}{\mu} |\log \varepsilon |$. Especially, $Y^\varepsilon$ equals to $Y$ if $x \in \mathbb{R} ^d \backslash D$. Thus we only consider the case that $x \in D$.

\begin{lem}
\label{lem41}
If $p\in \mathbb{N}$ satisfies $p > \frac{d}{4}$, then the estimate
\begin{align}
E\left [ \| Y^\varepsilon (\tau_\varepsilon \wedge T_\varepsilon ,\xi ,\cdot ) - Y(\tau_\varepsilon \wedge T_\varepsilon ,\xi ) \|_{L^{2p}(D)} ^{2p} \right ] \leq O(\varepsilon ^{2\gamma +1 + \kappa (2p-2) -\frac{2C_f p}{\mu}} |\log \varepsilon |)
\end{align}
holds for sufficiently small $\varepsilon >0$.
\end{lem}
\begin{proof}
From Ito's rule and easy computations, we obtain an estimate as below:
\begin{align}
(Y^\varepsilon &(\tau \wedge \tau_{\varepsilon} ^p , \xi ,x) - Y (\tau \wedge \tau_{\varepsilon} ^p , \xi) )^{2p}  \nonumber \\
&\leq 2 p \int _0 ^{\tau }(Y^\varepsilon(s \wedge \tau_{\varepsilon} ^p , \xi ,x) - Y(s \wedge \tau_{\varepsilon} ^p , \xi ))^{2p-1} (f(Y^\varepsilon(s \wedge \tau_{\varepsilon} ^p , \xi ,x)) - f(Y(s \wedge \tau_{\varepsilon} ^p , \xi ))) ds \nonumber \\
&+ 2p\varepsilon ^{\gamma +  \frac{1}{2}} \int _0 ^{\tau \wedge \tau_{\varepsilon} ^p }(Y ^\varepsilon (s , \xi ,x) - Y(s , \xi ) )^{2p-1} dW_s (x) \nonumber \\
&+ p(2p-1) \varepsilon ^{2\gamma +1} Q_d(x,x) \int _0 ^{\tau }(Y^\varepsilon(s \wedge \tau_{\varepsilon} ^p , \xi ,x) - Y(s \wedge \tau_{\varepsilon} ^p , \xi ))^{2p-2} ds \nonumber \\
&\leq 2C_f p \int _0 ^{\tau }(Y^\varepsilon(s \wedge \tau_{\varepsilon} ^p , \xi ,x) - Y(s \wedge \tau_{\varepsilon} ^p , \xi ))^{2p} ds \nonumber \\
&+ 2p\varepsilon ^{\gamma +  \frac{1}{2}} \int _0 ^{\tau \wedge \tau_{\varepsilon} ^p }(Y ^\varepsilon (s , \xi ,x) - Y(s , \xi ) )^{2p-1} dW_s (x) \nonumber \\
&+ p(2p-1) \varepsilon ^{2\gamma +1} Q_d(x,x) \int _0 ^{\tau }(Y^\varepsilon(s \wedge \tau_{\varepsilon} ^p , \xi ,x) - Y(s \wedge \tau_{\varepsilon} ^p , \xi ))^{2p-2} ds .
\end{align}
For the second inequality, we use the fact that $Y^\varepsilon$ never goes out of some interval until the time $\tau_\varepsilon$, because we can use Sobolev embedding $W^{1,2p} \hookrightarrow L^\infty$ to $Y^\varepsilon -Y$ for $2p > \frac{d}{2}$, and $Y$ is bounded. We apply Fubini's theorem to the integral $\int _\mathbb{R} \cdot dx$ and $E[\cdot ]$ and obtain
\begin{align}
E[\| Y^\varepsilon (\tau \wedge \tau_{\varepsilon} , \xi ,x) &- Y (\tau \wedge \tau_{\varepsilon} , \xi) \|_{L^{2p}} ^{2p}]  \nonumber \\
&\leq 2C_f p \int _0 ^{\tau }E[\|Y^\varepsilon(s \wedge \tau_{\varepsilon} , \xi ,x) - Y(s \wedge \tau_{\varepsilon} , \xi ) \|_{L^{2p}} ^{2p}]ds \nonumber \\
&+ p(2p-1) \varepsilon ^{2\gamma +1 + \kappa (2p-2)}\|Q_d(\cdot,\cdot)\|_{L^1} \tau .
\end{align}
By applying Gronwall's inequality to $E[\| Y^\varepsilon (\tau \wedge \tau_{\varepsilon} , \xi ,\cdot ) - Y_\tau (\tau \wedge \tau_{\varepsilon} , \xi) \|_{L^{2p}} ^{2p}]$, we get
\begin{align}
\label{est4-1}
E[\| Y^\varepsilon (T_\varepsilon \wedge \tau_{\varepsilon} , \xi ,\cdot ) &- Y (T_\varepsilon \wedge \tau_{\varepsilon} , \xi) \|_{L^{2p}} ^{2p}] \nonumber \\
&\leq p(2p-1)C \varepsilon ^{2\gamma +1 + \kappa (2p-2) -\frac{2C_f p}{\mu}}\|Q_d(\cdot,\cdot)\|_{L^1} |\log \varepsilon |.
\end{align}
This completes the proof of this lemma.
\end{proof}

We define the derivative of $Y^\varepsilon -Y$ with respect to $x\in D$. We set $Z_i ^\varepsilon(\tau, \xi ,x) := \frac{\partial}{\partial x_i } Y^\varepsilon$ ($i=1, \cdots, d$) and can show that the process obeys the SDE;
\begin{align}
\label{sde2}
\begin{cases}
\dot{Z} _i ^\varepsilon (\tau, \xi ,x) = f'( Y ^\varepsilon (\tau, \xi ,x ))Z _i ^\varepsilon (\tau, \xi ,x) + \varepsilon ^{\gamma + \frac{1}{2}} \frac{\partial }{\partial x_i} \dot{W}_\tau (x),\ \ \  \tau \in (0,\infty ),\\
Z _i ^\varepsilon (0, \xi ,x)=0.
\end{cases}
\end{align}
from the SDE (\ref{sde}). Here the symbol $\frac{\partial }{\partial x_i} \dot{W}_\tau (x)$ denotes formal derivative in time $t$ of $\frac{\partial }{\partial x_i} W_\tau (x)$ which means the derivative of $W_\tau (x)$ in a space variable $x_i$. Especially the derivative $\frac{\partial }{\partial x_i} W_\tau (x)$ equals to $\sqrt{Q_i (x,x)} W_\tau$ in law sense where $Q_i (x,y):= \frac{\partial }{\partial x_i}\frac{\partial }{\partial y_i}Q_d(x,y)$ and $W_\tau$ is a one-dimensional Brownian motion.

\begin{lem}
\label{lem42}
If $p\in \mathbb{N}$ satisfies $p > \frac{d}{4}$, then the estimate
\begin{align}
E\left [ \sum _{i=1} ^{d} \| Z_i ^\varepsilon (\tau_\varepsilon \wedge T_\varepsilon ,\xi ,\cdot ) \|_{L^{2p}(D)} ^{2p} \right ] \leq O(\varepsilon ^{2\gamma +1 + \kappa (2p-2) -\frac{2C_f p}{\mu}} |\log \varepsilon |)
\end{align}
holds for sufficiently small $\varepsilon >0$.
\end{lem}
\begin{proof}
Applying Ito's rule to $(Z_i ^\varepsilon )^{2p}$, we get
\begin{align}
(Z_i ^\varepsilon (\tau \wedge \tau_{\varepsilon} , \xi ,x) )^{2p}  &= 2p \int _0 ^{\tau \wedge \tau_{\varepsilon}}  f' (Y^\varepsilon (s , \xi ,x)) (Z_i ^\varepsilon(s , \xi ,x) )^{2p} ds \nonumber \\
&+ 2p\varepsilon ^{\gamma + \frac{1}{2}} \int _0 ^{\tau \wedge \tau_{\varepsilon}}Z_i ^\varepsilon (s , \xi ,x)^{2p-1} d\frac{\partial }{\partial x_i}W_s (x) \nonumber \\
&+ p(2p-1)\varepsilon ^{2\gamma +1} Q_i(x,x) \int _0 ^{\tau \wedge \tau_{\varepsilon}} (Z_i ^\varepsilon(s , \xi ,x) )^{2p-2} ds.
\end{align}
The rest of the proof is similar to that of Lemma \ref{lem41}.
\end{proof}

\begin{pro}
\label{thm41}
If there exists $p\in \mathbb{N}$ which satisfies $p > \frac{d}{4}$ and $2\gamma +1 + \kappa (2p-5) -(\alpha +p)\frac{2C_f }{\mu} >0$, then there exists $C>0$ such that
\begin{align}
P\left ( \|Y^\varepsilon - Y \|_\infty \leq C\varepsilon ^\kappa \right ) \to 1
\end{align}
as $\varepsilon \to 0$, where the norm $\| \cdot \|_\infty$ is a supremum norm on $\tau \in [0,T_\varepsilon ]$, $\xi \in [\varepsilon ^\alpha ,2C_0]$ and $x\in D$.
\end{pro}
\begin{proof}
Using Chebyschev inequality, we obtain the estimate
\begin{align}
P(\tau _\varepsilon \leq T_\varepsilon ) &\leq \varepsilon ^{-2\kappa} E[\| Y^\varepsilon ( \tau _\varepsilon \wedge T_\varepsilon ,\xi ,\cdot ) - Y ( \tau _\varepsilon \wedge T_\varepsilon ,\xi ) \|_{W^{1,2p}} ^2] \nonumber \\
&\leq C \varepsilon ^{2\gamma +1 + \kappa (2p-2) -\frac{2C_f p}{\mu} -2\kappa} | \log \varepsilon | \left \{ \|Q_d(\cdot ,\cdot )\| _{L^1} + \sum _{i=1} ^d \|Q_i(\cdot ,\cdot )\| _{L^1} \right \}
\end{align}
from Lemma \ref{lem41} and \ref{lem42}. There exists a positive constant $C' >0$ such that
\begin{align}
\label{est4-2}
P& \left( \underset{\tau \in [0,T_\varepsilon]}{\sup}  \| Y^\varepsilon ( \tau ,\xi ,\cdot ) - Y ( \tau ,\xi ) \|_{L^\infty} \leq C\varepsilon ^\kappa \right) \nonumber \\
&\geq P\left( \underset{\tau \in [0,T_\varepsilon]}{\sup}  \| Y^\varepsilon ( \tau ,\xi ,\cdot ) - Y ( \tau ,\xi ) \|_{W^{1,2p}} \leq \varepsilon ^\kappa \right) = P(\tau _\varepsilon \geq T_\varepsilon ) \nonumber \\
&\geq 1-C' \varepsilon ^{2\gamma +1 + \kappa (2p-4) -\frac{2C_f p}{\mu} } | \log \varepsilon |  \left \{ \|Q_d(\cdot ,\cdot )\| _{L^1} + \sum _{i=1} ^d \|Q_i(\cdot ,\cdot )\| _{L^1} \right \}
\end{align}
each fixed $\xi \in [\varepsilon ^{\alpha},2C_0]$. We use Sobolev embedding in the first line and the positive constant $C$ comes from Sobolev inequality. We note that the norm $\| \cdot \|_{L^\infty}$ means the supremum norm on $x \in D$. We need to estimate the supremum which is concerned with $\xi \in [\varepsilon ^{\alpha},2C_0]$. We define events
\begin{align}
\Omega_\xi ^\varepsilon := \{ \omega \in \Omega | \sup _{\tau \in [0,T_\varepsilon]}  \| Y^\varepsilon ( \tau ,\xi ,\cdot )(\omega ) - Y ( \tau ,\xi ) \|_{L^\infty} \leq C\varepsilon ^\kappa \}.
\end{align}
At first, we fix the path $\omega \in \Omega_\xi \cap \Omega_{\xi +\varepsilon ^{\kappa +\theta}}$ where $\theta := \frac{2C_f \tilde{\alpha}}{\mu}$ for small $\tilde{\alpha}  >\alpha$, and the initial value $\xi$ and $\xi +\varepsilon ^{\kappa +\theta}$ is in the interval $[\varepsilon ^{\alpha},2C_0]$. For all $\tau \in [0,T_\varepsilon]$, $\xi ' \in [\xi ,\xi +\varepsilon ^{\kappa +\theta}]$ and $x \in D$, we can compare as below.
\begin{align}
&Y^\varepsilon (\tau , \xi ,x) (\omega ) \leq Y^\varepsilon (\tau , \xi ' ,x) (\omega ) \leq Y^\varepsilon (\tau , \xi + \varepsilon ^{\kappa +\theta},x) (\omega ), \\
&Y (\tau , \xi ) \leq Y (\tau , \xi ' ) \leq Y (\tau , \xi + \varepsilon ^{\kappa +\theta} ).
\end{align}
These inequalities are from the comparison of initial value for SDE and ODE respectively. From (\ref{est4-2}) and Corollary \ref{cor31}, we can show that
\begin{align}
|Y^\varepsilon ( \tau ,\xi' ,x )(\omega ) - Y ( \tau ,\xi' )| \leq 3C\varepsilon ^{\kappa}
\end{align}
holds for all $\tau \in [0,T_\varepsilon ]$, $\xi' \in  [\xi ,\xi +\varepsilon ^{\kappa +\theta}]$, $x \in D$ and $\omega \in \Omega_\xi \cap \Omega_{\xi +\varepsilon ^{\kappa +\theta}}$. Repeating these argument, we get
\begin{align}
\label{est4-3}
P&\left( \|Y^\varepsilon - Y \|_\infty \leq C\varepsilon ^{\kappa} \right ) \nonumber \\
&\geq 1-C' \varepsilon ^{2\gamma +1 + \kappa (2p-4) -\frac{2C_f p}{\mu} } | \log \varepsilon | \left ( \left [ \frac{2C_0}{\varepsilon ^{\kappa +\theta}} \right ] +1 \right )  \left \{ \|Q_d(\cdot ,\cdot )\| _{L^1} + \sum _{i=1} ^d \|Q_i(\cdot ,\cdot )\| _{L^1} \right \},
\end{align}
where $\| \cdot \|_\infty$ is a supremum norm on $\tau \in [0,T_\varepsilon ]$, $\xi \in [\varepsilon ^\alpha ,2C_0]$ and $x\in D$, and thus the left hand side of (\ref{est4-3}) tends to 0 as $\varepsilon \to 0$.
\end{proof}

We note that we can remove the constant $C>0$ from Sobolev inequality. We can assert the same statements as Proposition \ref{thm41} in the case of $\xi \in [-2C_0,-\varepsilon ^{\alpha}]$. To sum up these results, we state next proposition.

\begin{pro}
\label{thm43}
If there exists $p\in \mathbb{N}$ which satisfies $p > \frac{d}{4}$ and $2\gamma +1 + \kappa (2p-5) -(\alpha +p)\frac{2C_f }{\mu} >0$, then we have that
\begin{align}
P\left ( \|Y^\varepsilon - Y \|_\infty \leq \varepsilon ^\kappa \right ) \to 1
\end{align}
as $\varepsilon \to 0$, where $\| \cdot \|_\infty$ is a supreme norm on $\tau \in [0,\frac{1}{\mu}|\log \varepsilon |]$, $\xi \in [-2C_0,2C_0]\backslash (-\varepsilon ^{\alpha}, \varepsilon ^{\alpha})$ and $x\in D$.
\end{pro}

If we set a stopping time 
\begin{align}
\tau _2 = \inf \{ \tau >0 | |Y^\varepsilon - Y | > \varepsilon ^\kappa \ for \ some \ \xi \in [-2C_0,2C_0]\backslash (-\varepsilon ^{\alpha}, \varepsilon ^{\alpha}),\  x\in D \} ,
\end{align}
then the statement of Proposition \ref{thm43} is equivalent to $\lim_{\varepsilon \to 0}P(\tau _2 > \frac{1}{\mu}|\log \varepsilon |) = 1$. Now we consider the time $t \leq \varepsilon \tau _2$ set functions as below:
\begin{align}
w_\varepsilon ^\pm (t,x) = Y^\varepsilon \left( \frac{t}{\varepsilon } , u_0  ^{\pm} (x) \pm \varepsilon C_2 \left(e^\frac{\mu t}{\varepsilon}-1 \right) ,x \right) .
\end{align}
Here the function $u_0 ^{\pm}$ can be taken satisfying (\ref{eq:ini2}) the Neumann boundary condition on $\partial D$, $u_0 ^{-}\leq u_0\leq u_0 ^{+}$ and $u_0 ^{-}= u_0= u_0 ^{+}$ if $x\in D$ and $d(x,\partial D)\geq d_1$ for sufficiently small $d_1 >0$. This is discussed in Section 3.3 of \cite{ham}.

\section{Smooth approximation}
\label{sec5}
Now we need to prove that $w_\varepsilon ^\pm (t,x)$ are super and sub solutions of (\ref{eq:spde}). However, because $\dot{W}_t (x)$ is the formal derivative in time $t$ and singular in time, we cannot directly use the comparison theorem of PDE for each path of the solutions. Thus we approximate these solutions smoothly, and use the comparison theorem to these smooth solutions $u^{\varepsilon , \delta}(t,x)$ and $w_{\varepsilon  , \delta}^\pm (t,x)$. We mention the definition of $u^{\varepsilon , \delta}(t,x)$ and $w_{\varepsilon  , \delta}^\pm (t,x)$ later. We consider a smooth approximation $W_t ^{(\delta )}(x)$ of $W_t ^{Q_d}(x)$. First, we define $W_t ^\delta(x)$ for $\delta >0$ as a convolution
\begin{align}
W_t ^\delta (x) := \int_0 ^t W_s ^{Q_d} (x) \rho _\delta (t-s)ds,\ \ \ t \geq 0,\ x\in D
\end{align}
where $\rho _\delta (t):= \frac{1}{\delta}\rho (\frac{t}{\delta})$, and $\rho$ satisfies the following conditions.
\begin{align}
 \begin{cases}
  \text{(i)}\int_\mathbb{R} \rho (t)dt=1, & \\
  \text{(ii)}\rm{supp} \rho \subset [0,1], & \\
  \text{(iii)}\rho \in C^\infty (\mathbb{R} ). \\
 \end{cases}
\end{align}
Then, we can take $W_t ^{(\delta )}(x)( \omega )$ which converges to $W_t ^{Q_d}(x)(\omega )$ uniformly on $[0,\infty) \times \Omega\ (a.s.)$ as $\delta \to 0$ by changing the speed of convergence $\delta$ in $W_t ^\delta (x)$ for each $\omega$. 

\begin{lem}
\label{lem51}
For every $\alpha \in ( 0, \frac{1}{2})$ and $\beta \in ( 0, \infty )$ there exists almost surely finite and positive random variable $K(\omega )$ such that
\begin{align}
|W_t ^{Q_d}(x)(\omega ) - W_s ^{Q_d}(y)(\omega ) | \leq K( \omega )\{ |t-s|^\alpha+ |x-y| ^\beta \} , \ \ \  t,s \in [0,1],\ x\in D
\end{align}
holds almost surely. 
\end{lem}
\begin{proof}
From the differentiability and the boundedness of ${Q_d}(x,x)$, we can easily show that
\begin{align}
E[(W_t ^{Q_d}(x) - W_s ^{Q_d}(y))^2] \leq C\{  |t-s|+ |x-y|^\beta \}
\end{align}
for all $\beta >0$. Because $\{ W_t ^{Q_d}(x)\}$ is a Gaussian system, we can apply Kolmogorov-Totoki's criterion and show the claim of this lemma. See Lemma 2.3 and Corollary 2.1 of \cite{k} for the similar argument.
\end{proof}

For any $\alpha \in ( 0, \frac{1}{2})$ and $x=y$ in Lemma \ref{lem51}, easy computations give us
\begin{align}
|W_t ^{Q_d}(x) - W_t ^\delta(x) | &= \left| \int_0 ^t W_t ^{Q_d}(x) \rho _\delta (t-s)ds - \int_0 ^t W_s ^{Q_d}(x) \rho _\delta (t-s)ds\right| \nonumber \\
&\leq \int_0 ^t |W_t ^{Q_d}(x) - W_s ^{Q_d}(x)| \rho _\delta (t-s)ds \leq K( \omega ) \delta ^\alpha .
\end{align}
Thus we set $\delta (\omega ):= (\frac{\delta }{K(\omega )})^{\frac{1}{\alpha }}$ and define $W_t ^{(\delta )}(\omega ) := W_t ^{\delta(\omega )}(x)(\omega )$, then $W_t ^{(\delta )}(x)$ is the measurable process and we have that $\sup _{t\in [0,\infty )}|W_t ^{Q_d} (x)- W_t ^{(\delta)}(x) | \leq \delta$ almost surely. Now we set the random variable $K(\omega)$ large enough $\sup _{t\in [0,\infty )}|\partial _x W_t ^{Q_d} (x)- \partial _x W_t ^{(\delta)}(x) | \leq \delta$ and $\sup _{t\in [0,\infty )}|\partial _{xx}W_t ^{Q_d} (x)- \partial _{xx}W_t ^{(\delta)}(x) | \leq \delta$ to hold. Now we change the variable $t$ to $\varepsilon \tau$ and consider a solution of the following ODE:
\begin{align}
\label{ode5-1}
\begin{cases}
\dot{Y} ^{\varepsilon , \delta} (\tau, \xi ,x) = f( Y ^{\varepsilon , \delta} (\tau, \xi ,x)) + \varepsilon ^{\gamma + \frac{1}{2}} \dot{\widetilde{W}}_\tau ^{(\delta )}(x),\ \ \  \tau >0 ,\\
Y ^{\varepsilon , \delta} (0, \xi ,x)=\xi \in [-2C_0 ,2C_0].
\end{cases}
\end{align}
For simplicity, we denote the process $\widetilde{W}_\tau ^{(\delta )}(x):=\varepsilon ^{-\frac{1}{2}}W_{\varepsilon \tau} ^{(\delta )}(x) $ as ${W}_\tau ^{(\delta )}(x)$ in (\ref{ode5-1}). And we set
\begin{align}
w_{\varepsilon ,\delta} ^\pm (t,x) = Y^{\varepsilon ,\delta} \left( \frac{t}{\varepsilon } , u_0 ^\pm (x)  \pm \varepsilon C_2 \left(e^\frac{\mu t}{\varepsilon}-1 \right) ,x \right) .
\end{align}

\begin{lem}
\label{lem52}
We set a stopping time
\begin{align}
\tau _3 := \inf \{ \tau >0 | |Y ^{\varepsilon ,\delta} (\tau ,\xi ,x)|>2C_0+ \delta ' \ for\ some\ \xi \in [-2C_0,2C_0],\ x \in D \ or \ \delta \in (0,\delta _\varepsilon ] \},
\end{align}
for fixed $\delta '>0$. If $2\gamma +1 > 2C_f C_1$, then for every sequence $\{\delta _\varepsilon \}$ which tends to 0 as $\varepsilon \to 0$, we have that $\lim _{\varepsilon \to 0} P(\tau _3 > C_1| \log \varepsilon |) =1$.
\end{lem}
\begin{proof}
Unless $x\in D$, $Y ^{\varepsilon ,\delta}$ equals to the solutions of ODE $Y$, and thus we only consider the case that $x\in D$. Now we set
\begin{align}
\tilde{\tau}  _3 := \inf \{ \tau >0 | \|Y ^{\varepsilon ,\delta} (\tau ,\xi ,\cdot )\|_{W^{1,2p}}>C^{-1} (2C_0+ \delta ') \ for\ some\ \xi \in [-2C_0,2C_0] \ or \ \delta \in (0,\delta _\varepsilon ] \},
\end{align}
where the positive constant $C>0$ comes from Sobolev's inequality $\| \varphi \|_{L^\infty} \leq C \| \varphi \|_{W^{1,2p}}$ for $\varphi \in W^{1,2p}(D)$ for $p>\frac{d}{4}$. We note that $\tilde{\tau}  _3 \leq \tau _3$ almost surely. At first, we get
\begin{align}
|Y_{\tau \wedge \tilde{\tau} _3} ^{\varepsilon ,\delta} | & \leq C_f \int _0 ^{\tau \wedge \tilde{\tau} _3} |Y_s ^{\varepsilon ,\delta} | ds + \varepsilon ^{\gamma +\frac{1}{2}} |W_{\tau \wedge \tilde{\tau} _3} ^{(\delta)} (x)| ,
\end{align}
and Gronwall's inequality and definition of $W_\tau ^{(\delta)} (x)$ give us
\begin{align}
|Y_{\tau \wedge \tilde{\tau} _3} ^{\varepsilon ,\delta} | \leq \exp (C_f \tau )  \varepsilon ^{\gamma +\frac{1}{2}}\left( |W_{\tau \wedge \tilde{\tau} _3} (x)| + \delta _\varepsilon \right) ,
\end{align}
if $\delta \in (0,\delta _\varepsilon ]$. In the same way, we obtain
\begin{align}
|Z_{i} ^{\varepsilon ,\delta} (\tau \wedge \tilde{\tau} _3) | \leq \exp (C_f \tau )  \varepsilon ^{\gamma +\frac{1}{2}}( |\partial _{x_i} W_{\tau \wedge \tilde{\tau} _3} (x)| + \delta _\varepsilon ) .
\end{align}
Here we define $Z _i ^{\varepsilon ,\delta}$ similarly to $Z_i ^\varepsilon$ in Section \ref{sec4}. To sum up these estimates, we get
\begin{align}
&E[ \|Y_{\tau \wedge \tilde{\tau} _3} ^{\varepsilon ,\delta} \|_{W^{1,2p}} ^{2p}] \nonumber \\
&\leq \exp (2pC_f \tau )  \varepsilon ^{p(2\gamma +1)}C\left \{ \left ( \|Q_d(\cdot ,\cdot )\|_{L^{2p}} ^{2p}+\sum _{i=1} ^d \|Q_i(\cdot ,\cdot )\|_{L^{2p}} ^{2p} \right ) \tau + \delta _\varepsilon ^{2p} \right \} ,
\end{align}
for some constant $C>0$ by using an inequality $(a+b)^{2p} \leq 2^{2p} (a^{2p} + b^{2p})$. From Chebyshev's inequality,
\begin{align}
&P(\tau _3 \leq C_1| \log \varepsilon |)  \leq P(\tilde{\tau} _3 \leq C_1| \log \varepsilon |) \leq C^2(2C_0+ \delta ')^{-2} E[ \|Y_{C_1| \log \varepsilon | \wedge \tilde{\tau} _3} ^{\varepsilon ,\delta} \|_{W^{1,2p}} ^{2p}] \nonumber \\
& \leq C^2(2C_0+ \delta ')^{-2}  \varepsilon ^{p(2\gamma +1 - 2C_f C_1)}\left \{ \left ( \|Q_d(\cdot ,\cdot )\|_{L^{2p}} ^{2p}+\sum _{i=1} ^d \|Q_i(\cdot ,\cdot )\|_{L^{2p}} ^{2p} \right ) \tau + \delta _\varepsilon ^{2p} \right \}
\end{align}
completes the proof of the claim.
\end{proof}

\begin{pro}
\label{thm51}
If we set $\tau _4:= \tau_2 \wedge \tau _3 $, then there exists a sequence $\{\delta _\varepsilon \}$ which tends to 0 as $\varepsilon \to 0$, and
\begin{align}
\label{est5-3}
\underset{x \in D}{\sup} \underset{\tau \in [0,C_1| \log \varepsilon | \wedge \tau _4 ]}{\sup} \underset{\xi \in [-2C_0,2C_0]}{\sup} \left| Y^{\varepsilon , \delta} (\tau ,\xi ,x ) -Y^\varepsilon (\tau ,\xi ,x ) \right| \leq \varepsilon^{\gamma -C_f C_1}  \delta \ for \ all \ \delta \in (0,\delta_\varepsilon]
\end{align}
holds $P$-a.s.
\end{pro}
\begin{proof}
At first, we note that $P(\tau _4 \geq C_1| \log \varepsilon |) \to 1$ as $\varepsilon \to 0$ from Section \ref{sec4} and Lemma \ref{lem52}. And we have that
\begin{align}
\left| Y_{\tau \wedge \tau _4} ^{\varepsilon , \delta} -Y_{\tau \wedge \tau _4} ^\varepsilon \right| & \leq \int _0 ^{\tau \wedge \tau _4} \left| f(Y_s ^{\varepsilon , \delta}) -f(Y_s ^\varepsilon ) \right| ds +\varepsilon ^{\gamma + \frac{1}{2}}|W_\tau ^{(\delta)}(x) - W_\tau (x) | \nonumber \\
&\leq C_f \int _0 ^\tau \left| Y_{s \wedge \tau _4} ^{\varepsilon , \delta}  -Y_{s \wedge \tau _4} ^\varepsilon \right| ds +\varepsilon ^{\gamma + \frac{1}{2}}\delta
\end{align}
for each $x \in D$ and $\xi \in [-2C_0,2C_0]$ because $Y^\varepsilon$ and $Y^{\varepsilon , \delta}$ do not go out of $[-2C_0-\delta ',2C_0 +\delta ']$ until the time $\tau _4$. From Gronwall's inequality, for all $\tau \in [0, C_1| \log \varepsilon | ]$ we have that
\begin{align}
\label{est5-2}
\left| Y ^{\varepsilon , \delta} (\tau \wedge \tau _4 , \xi ,x ) -Y ^\varepsilon (\tau \wedge \tau _4 , \xi ,x ) \right| \leq \varepsilon^{\gamma -C_f C_1}  \delta ,
\end{align}
for each $x \in D$ and $\xi \in [-2C_0,2C_0]$.
\end{proof}

We define $u^{\varepsilon , \delta}$ as a solution of PDE:
\begin{align}
\begin{cases}
\dot{u} ^{\varepsilon , \delta} (t,x) &= \Delta u^{\varepsilon , \delta} (t,x)+\displaystyle{\frac{1}{\varepsilon}} f(u^{\varepsilon , \delta} (t,x) ) + \varepsilon^{\gamma} \dot{W} ^{(\delta )} _t(x),\ \ \ t> 0,\  x\in D,\\
u^{\varepsilon , \delta} (0,x) &= u_0 (x),\ \ \ x\in D, \\
\frac{\partial u}{\partial \nu} (t,x) &= 0,\ \ \ t> 0,\  x\in \partial D
\end{cases}
\end{align}
Now we prove the similar estimate to (\ref{est5-3}) for $u^{\varepsilon , \delta}$. 

\begin{lem}
\label{lem53}
If we set a stopping time
\begin{align}
\tau _5 :=\inf  \{t>0 | |u^{\varepsilon , \delta} (t,x)| > 2C_0\ for \ some \ x\in D \ or \ \delta \in (0,\delta_\varepsilon] \},
\end{align}
then there exists a sequence $\{\delta _\varepsilon \}$ which tends to 0 as $\varepsilon \to 0$, such that $\lim _{\varepsilon \to 0} P(\tau _5 > \varepsilon ^{-n}) =1$ for all $n \in \mathbb{N}$.
\end{lem}
\begin{proof}
We can prove the claim by modifying the proof of Proposition \ref{thm31}. Indeed, from the integration by parts for the normal integral and the stochastic integral, we get
\begin{align}
\label{est6-7}
&\left | \int_{D} \int_0^t  p(t-s,x,y)dW_s ^{(\delta )}(y) dy -\int_{D} \int_0^tp(t-s,x,y)dW_s (y) dy \right | \nonumber \\
&\ \ \ \ \ = \left | (W_t ^{(\delta )}(x) - W_t(x)) - \int _{D}  \int _0 ^t \partial _s p(t-s,x,y)(W_s ^{(\delta )}(y) - W_s (y))dsdy \right | \nonumber \\
&\ \ \ \ \ = \left | (W_t ^{(\delta )}(x) - W_t(x)) + \int _{D}  \int _0 ^t \Delta _y p(t-s,x,y)(W_s ^{(\delta )}(y) - W_s (y))dsdy \right | \nonumber \\
&\ \ \ \ \ = \left | (W_t ^{(\delta )}(x) - W_t(x)) + \int _{D}  \int _0 ^t p(t-s,x,y) \Delta _y(W_s ^{(\delta )}(y) - W_s (y))dsdy \right | \nonumber \\
&\ \ \ \ \ \leq \delta + \int _{D}  \int _0 ^t p(t-s,x,y) dsdy \delta =(t+1)\delta ,
\end{align}
where $p(t,x,y)$ defines the heat semigroup $e^{\Delta t}$ on $L^2 (D)$ as $(e^{\Delta t} \varphi ) (x) := \int _D \varphi (y) p(t,x,y) dy$. We use the condition of the heat kernel in the third line and use Green's divergence theorem in the forth line. Indeed, from Green's theorem, the integration on the boundary $\partial D$
\begin{align}
\int _{\partial D}  \int _0 ^t \left \{ \partial _\nu p(t-s,x)(W_s ^{(\delta )} - W_s ) -  p(t-s,x)\partial _\nu (W_s ^{(\delta )} - W_s ) \right \} dsd\nu
\end{align}
appears where $\nu$ is the exterior unit normal vector on $\partial D$. However, this term is vanished because of the Neumann boundary condition of $p(t,x,\cdot)$ and the condition supp$Q \subset D\times D$. The computation in the second line comes from the fact that $\int _D p(t,x,y)\varphi(y)dy \to \varphi (x)$ as $t\to 0$ for all $\varphi \in C_0^{\infty} (D)$. We note that for each $\omega \in \Omega$ and $t \in [0,\infty)$, $W_t ^{(\delta )}(x)$ and $W_t (x)$ are in the class of $C_0^{\infty} (D)$. These estimates imply that the stochastic convolution term never harm in the proof of Proposition \ref{thm31} if we take $\delta _\varepsilon \ll \varepsilon ^n$.
\end{proof}

We only need to see the time $C_1 \varepsilon |\log \varepsilon |$ in this section. However, we need a same estimate of longer time in Section \ref{sec7}. Thus we take the time $\varepsilon ^{-n}$ which is much longer than $C_1 \varepsilon |\log \varepsilon |$ in Lemma \ref{lem53}.

\begin{pro}
\label{thm52}
If we set $\tau _6:= \tau_1 \wedge \tau _5 $, then for every sequence $\{\delta _\varepsilon \}$ which tends to 0 as $\varepsilon \to 0$ and
\begin{align}
\underset{t \in [0,C_1 \varepsilon |\log \varepsilon | \wedge \tau _6 ]}{\sup}\underset{x \in D}{\sup}|u^{\varepsilon , \delta} (t,x) -u^\varepsilon (t,x)| \leq \varepsilon ^{1+ \gamma- C_f C_1} | \log \varepsilon | \delta \ for \ all \ \delta \in (0,\delta_\varepsilon]
\end{align}
holds $P$-a.s.
\end{pro}
\begin{proof}
At first, we consider the mild form of $u^{\varepsilon , \delta} (t,x) (\omega ) -u^\varepsilon (t,x)   (\omega )$. Then, since the initial values are same, we get
\begin{align}
\label{est5-5}
u^{\varepsilon , \delta} (t,x) -u^\varepsilon (t,x) = \frac{1}{\varepsilon} \int _0 ^t \int _D &p(t-s,x,y)\{ f(u^{\varepsilon , \delta}(s,y))-f(u^{\varepsilon}(s,y)) \} dyds\nonumber \\
&+ \varepsilon ^{\gamma}\int_D \int_0^tp(t-s,x,y)dW_s ^{(\delta )}(y) dy\nonumber \\
&- \varepsilon ^{\gamma}\int_D \int_0^tp(t-s,x,y)dW_s(y) dy
\end{align}
where $p(t,x,y)$ is defined in the proof of Lemma \ref{lem53}. From the estimate (\ref{est6-7}) in the proof of Lemma \ref{lem53} and the definition of $\tau _6$, we get
\begin{align}
\label{est5-6}
|u^{\varepsilon , \delta} (t,x) - & u^\varepsilon (t,x)| \leq \frac{1}{\varepsilon} \int _0 ^t \int _D p(t-s,x,y)C_f |u^{\varepsilon , \delta}(s,y)-u^{\varepsilon}(s,y) | dyds + \varepsilon ^{\gamma} (t+1) \delta \nonumber \\
& \leq \frac{1}{\varepsilon} \int _0 ^t \int _D p(t-s,x,y)C_f \left( \underset{z \in D}{\sup}|u^{\varepsilon , \delta}(s,z)-u^{\varepsilon }(s,z) |\right) dyds + \varepsilon ^{\gamma} (t+1) \delta \nonumber \\
& \leq \frac{C_f}{\varepsilon} \int _0 ^t \left( \underset{x \in D}{\sup}|u^{\varepsilon , \delta}(s,x)-u^{\varepsilon}(s,x) |\right) ds + C_1 \varepsilon ^{1+ \gamma} | \log \varepsilon | \delta
\end{align}
for each $t \in [0,C_1 \varepsilon |\log \varepsilon | \wedge \tau _6 ]$, because the integration of the heat kernel $p(t,x,y)$ with respect to $y \in D$ is 1. If we take the supremum of the left hand side on $x\in D$, (\ref{est5-6}) becomes the form which we can apply Gronwall's inequality to $\sup _{x \in D}|u^{\varepsilon , \delta}(t,x)-u^{\varepsilon}(t,x) |$. And thus we obtain the following estimates by Gronwall's inequality.
\begin{align}
\underset{x \in D}{\sup}|u^{\varepsilon , \delta} (t,x) -u^\varepsilon (t,x)| &\leq C \varepsilon ^{1+ \gamma- C_f C_1} | \log \varepsilon | \delta .
\end{align}
In the same way as the proof of Proposition \ref{thm51}, we can prove the claim of this lemma.
\end{proof}

\section{Comparison arguments}
\label{sec6}
In this section, we apply the comparison theorem for PDE to $u^{\varepsilon , \delta}$ and $w_{\varepsilon , \delta} ^\pm$. And we prove that $u^{\varepsilon , \delta}$ and $w_{\varepsilon , \delta} ^\pm$ converge to $u^{\varepsilon }$ and $w_{\varepsilon } ^\pm$ uniformly from the results of Section \ref{sec5}. Our claim in this section is formulated in the following Theorem. The argument in this section is based on the methods of Alfaro et al \cite{ham}.

\begin{pro}
\label{thm61}
For every $0 < C_1 < \frac{1}{\mu}$,
\begin{align}
 w_{\varepsilon} ^- (t,x) \leq u^{\varepsilon} (t,x) \leq w_{\varepsilon } ^+ (t,x)\ for\ every\  t \in [0, C_1 \varepsilon |\log \varepsilon | \wedge \varepsilon \tau _4 \wedge \tau _6 \wedge \varepsilon \tau _7],\  x \in D
\end{align}
holds $P$-a.s.
\end{pro}

\subsection{Auxiliary estimates}
Before proving the proposition, we need to prove some lemmas as preparations in advance. We define the following stochastic process.
\begin{align}
\label{def5-1}
A^{\varepsilon , \delta} (\tau ,\xi ,x ) := \frac{Y_{\xi \xi } ^{\varepsilon , \delta} (\tau ,\xi ,x )}{Y_\xi ^{\varepsilon , \delta} (\tau ,\xi ,x )},
\end{align}
where $Y_\xi ^{\varepsilon , \delta}$ and $Y_{\xi \xi} ^{\varepsilon , \delta}$ mean the derivatives of $Y^{\varepsilon , \delta}$ with respect to $\xi$. Because $W_\tau ^{(\delta )} (x)$ is independent of $\xi$, we can compute $Y_\xi ^{\varepsilon , \delta}$ and $Y_{\xi \xi} ^{\varepsilon , \delta}$ from the form of the ODE (\ref{ode5-1}). We get an ODE
\begin{align}
\label{ode6-1}
\begin{cases}
{Y}_{\xi \tau} ^{\varepsilon , \delta} (\tau, \xi ,x) = {Y}_\xi ^{\varepsilon , \delta}(\tau, \xi ,x )  f'( Y ^{\varepsilon , \delta}(\tau, \xi ,x ) ),\ \ \ \tau > 0,\ x\in D \\
Y ^{\varepsilon , \delta} _\xi (0, \xi ,x)=1,
\end{cases}
\end{align}
and we obtain
\begin{align}
\label{eq5-1}
{Y}_\xi ^{\varepsilon , \delta} (\tau, \xi ,x) = \exp \left( \int _0 ^\tau f'\left( Y^{\varepsilon , \delta}(s,\xi ,x )\right) ds \right),\ \ \ \tau \geq0,\ x\in D
\end{align}
from (\ref{ode6-1}). In particular, ${Y}_\xi ^{\varepsilon , \delta}>0$ allows us to define $A ^{\varepsilon , \delta} (\tau, \xi ,x)$ as in (\ref{def5-1}). Now we get 
\begin{align}
A ^{\varepsilon , \delta} (\tau, \xi ,x) = \int _0 ^\tau Y_\xi ^{\varepsilon , \delta}(s,\xi ,x ) f''\left( Y^{\varepsilon , \delta}(s,\xi ,x )\right) ds,\ \ \ \tau \geq 0,\ x\in D
\end{align}
by computing $Y_{\xi \xi} ^{\varepsilon , \delta}$ from (\ref{eq5-1}). Now we derive estimates on $A ^{\varepsilon , \delta}$ and use these estimates for the proof of Proposition \ref{thm61}. We divide into two cases. At first, we prove the case that $\xi>C\varepsilon ^\alpha$ is less than 1.

\begin{lem}
\label{lem61}
For all $\eta \in (0,1)$, there exists a sequence $\{\delta _\varepsilon \}$ which tends to 0 as $\varepsilon \to 0$ and $C_3( \eta )>0$ such that
\begin{align}
\underset {\delta \in (0,\delta_\varepsilon )}{\sup} \underset {\xi \in (C \varepsilon ^\alpha,1-\eta)}{\sup} \underset {x \in D}{\sup} |A^{\varepsilon ,\delta} (\tau, \xi ,x)|\leq C_3 (e^{\mu \tau}-1) \ for\ all\ \tau \in [0,C_1 |\log \varepsilon | \wedge T_1 \wedge \tau _4] 
\end{align}
hold $P$-a.s. Here, let $T_1$ be the time defined by $T_1:= \inf \{ \tau \geq 0|Y(\tau ,\xi ) = 1-\eta \}$.
\end{lem}
\begin{proof}
We fix a time $\tau \in [0,C_1 |\log \varepsilon | \wedge T_1 \wedge \tau _4]$. By applying Lemma 3.4 in \cite{ham}, we can assert that there exist positive constants $\widetilde{C} _1(\eta ),\widetilde{C} _2(\eta )>0$ such that
\begin{align}
\label{est6-5}
\widetilde{C} _1 e^{\mu \tau} \leq Y_\xi (\tau ,\xi ) \leq \widetilde{C} _2 e^{\mu \tau}
\end{align}
before the time $T_1$. And we define $\delta _\varepsilon :=\exp(-\varepsilon ^{-\beta})$ for any $\beta >0$, then we obtain for all $\delta \in (0,\delta_\varepsilon )$, $\tau \in [0,C_1|\log \varepsilon |\wedge \tau _4]$, $\xi \in (C \varepsilon ^\alpha,1-\eta)$ and $x \in D$
\begin{align}
 |Y^{\varepsilon, \delta}(\tau ,\xi ,x ) -Y^{\varepsilon}(\tau ,\xi ,x)| \leq \varepsilon^{\gamma -C_f C_1} \delta \leq \varepsilon^{\gamma} \exp (-\varepsilon ^{-\beta}) \leq C \varepsilon ^\kappa
\end{align}
from (\ref{est5-3}) of Proposition \ref{thm51} for sufficiently small $\varepsilon>0$. Moreover the estimate in Proposition \ref{thm43} gives us
\begin{align}
\label{est6-4}
\underset{\tau \in [0,C_1|\log \varepsilon | \wedge \tau _4]}{\sup}\underset {\xi \in (C \varepsilon ^\alpha,1-\eta)}{\sup} \underset {x \in D}{\sup} |Y^{\varepsilon, \delta}(\tau ,\xi ,x ) -Y(\tau ,\xi )|\leq C\varepsilon ^\kappa.
\end{align}
And thus, we obtain the following estimates.
\begin{align}
\label{est6-6}
|Y_\xi ^{\varepsilon ,\delta}(\tau , \xi ,x )-Y_\xi (\tau , \xi )| & = 
Y_\xi (\tau , \xi )\left| \exp \left[ \int _0 ^\tau \left \{ f'( Y ^{\varepsilon ,\delta}(s , \xi ,x ) ) -f'( Y (s , \xi )) \right \} ds \right] -1 \right| \nonumber \\
& \leq Y_\xi (\tau , \xi )\left| \exp \left[ \int _0 ^\tau \left| f'( Y ^{\varepsilon ,\delta}(s , \xi ,x ) ) -f'( Y (s , \xi )) \right| ds \right] -1 \right| \nonumber \\
& \leq Y_\xi (\tau , \xi )\left| \exp (C\varepsilon ^{\kappa })-1 \right| \ \ \ ( \tau \in [0,C|\log \varepsilon | \wedge T_1 \wedge \tau _4]).
\end{align}
The first line is easily proved from (\ref{eq5-1}) because we can compute $Y_\xi$ in the same way as $Y_\xi ^{\varepsilon ,\delta}$. We used the inequality $|e^x -1|<|e^{|x|}-1|$ for the second line. Third line is from the differentiability of $f$, the behavior of $Y$ and $Y^{\varepsilon ,\delta}$ from (\ref{est6-4}). We use inequality $|e^a -1|<|e^b-1|\ (b\geq a\geq 0)$ in the second and the third line. If we take $\varepsilon$ sufficiently small, there exists $C' _1(\eta )>0$ which satisfies
\begin{align}
C' _1e^{\mu \tau} &\leq \widetilde{C}_1(1-|\exp (C\varepsilon ^{\kappa })-1|)e^{\mu \tau} \leq Y_\xi(\tau ,\xi )(1-|\exp (C\varepsilon ^{\kappa })-1|) \leq Y_\xi ^{\varepsilon ,\delta}(\tau ,\xi ,x)
\end{align}
for all $x\in D$, $\xi \in (C \varepsilon ^\alpha,1-\eta)$, $\tau \in [0,C_1 |\log \varepsilon | \wedge T_1 \wedge \tau _4]$ and $\delta \in (0,\delta _\varepsilon)$ from (\ref{est6-5}) and (\ref{est6-6}). In the same way, there exists $C' _2(\eta )>0$ such that
\begin{align}
 Y_\xi ^{\varepsilon ,\delta}(\tau ,\xi ,x ) \leq C' _2e^{\mu \tau}.
\end{align}
The rest of the proof is the same as Lemma 3.4 of \cite{ham}. Indeed, we get
\begin{align}
|A^{\varepsilon , \delta} (\tau ,\xi ,x)| \leq \| f'' \|_{L^\infty ([-1,1])}\int _0 ^\tau (C' _1 \vee C' _2 )e^{\mu s}ds \leq C_3(e^{\mu \tau} -1)
\end{align}
for some $C_3(\eta )$, and this completes the proof of the lemma.
\end{proof}

Next we consider the case of $\xi \in [1 - \eta,1+M]$.
\begin{lem}
\label{lem62}
For all $\eta \in (0,1)$ and $M>0$, there exist a sequence $\{\delta _\varepsilon \}$ which tends to 0 as $\varepsilon \to 0$ and a positive constant $C_4( \eta, M )>0$ such that
\begin{align}
\underset {\delta \in (0,\delta_\varepsilon )}{\sup} \underset {\xi \in [1-\eta ,1+M]}{\sup} \underset {x \in D}{\sup} |A^{\varepsilon ,\delta} (\tau, \xi ,x)| \leq C_4 \tau \ for\ all \  \tau \in [0,C_1 |\log \varepsilon | \wedge \tau _4]
\end{align}
holds $P$-a.s.
\end{lem}
\begin{proof}
Take $\varepsilon >0$ small enough and fix a time $\tau \in [0,C_1 |\log \varepsilon | \wedge \tau _4]$. From Lemma 3.6 of \cite{ham}, there exists $\widetilde{C}_4(\eta, M)>0$ such that $Y_\xi (\tau, \xi) \leq \widetilde{C}_4$ for all $\tau \geq 0$. Thus there exists positive constant $C'_4(\eta, M)>0$ such that
\begin{align}
\label{est6-7’}
Y_\xi ^{\varepsilon ,\delta}(\tau, \xi ,x) &= \exp \left \{ \int _0 ^\tau f'(Y_\xi ^{\varepsilon ,\delta}(s, \xi ,x) )ds \right \} \nonumber \\
&= Y_\xi (\tau, \xi) \exp \left [ \int _0 ^\tau \{ f'(Y_\xi ^{\varepsilon ,\delta}(s, \xi ,x) ) - f'(Y_\xi (s, \xi) ) \} ds \right ] \\
&\leq  \widetilde{C}_4 \exp (C\varepsilon ^{\kappa} \tau)\leq \widetilde{C}_4 \exp (C\varepsilon ^{\kappa}|\log \varepsilon |) \leq  C'_4.\nonumber
\end{align}
In the estimate of the third line, we use (\ref{est6-6}). From the behavior of $Y^{\varepsilon ,\delta}(\tau ,\xi )$ on $\Omega _\varepsilon$ and (\ref{est6-7’}), we now see that there exists $C_4(\eta, M)>0$ such that
\begin{align}
|A^{\varepsilon ,\delta}| \leq C_4' \int _0 ^\tau |f''(Y^{\varepsilon ,\delta})(s ,\xi )|ds \leq C_4 \tau .
\end{align}
This completes the proof.
\end{proof}

Next lemma is the combination of Lemmas \ref{lem61} and \ref{lem62}. We can prove in the same way as Lemma 3.7 of \cite{ham}. 

\begin{lem}
\label{lem63}
For all $\eta \in (0,1)$, there exist a sequence $\{\delta _\varepsilon \}$ which tends to 0 as $\varepsilon \to 0$ and a positive constant $C_5( \eta )>0$ such that
\begin{align}
\underset {\delta \in (0,\delta_\varepsilon )}{\sup} \underset {\xi \in (C\varepsilon ^\alpha ,2C_0]}{\sup} \underset {x \in D}{\sup}|A^{\varepsilon ,\delta} (\tau, \xi ,x)| \leq C_5 (e^{\mu \tau} -1) \ for\ all\  \tau \in [0,C_1|\log \varepsilon | \wedge \tau _4]
\end{align}
holds $P$-a.s.
\end{lem}
\begin{proof}
We fix $M$ as in Lemma \ref{lem62} large enough such that $[-2C_0,2C_0] \subset [-1-M,1+M]$, then the positive constant $C_4$ is independent of $M$. Because $\tau \leq \frac{1}{\mu}(e^{\mu \tau} -1)$ we can take $C_5$ as $C_5 \geq C_3 \vee \frac{C_4}{\mu}$.
\end{proof}

We can prove the case of $\xi \in [-2C_0,-C\varepsilon ^\alpha)$ in a same way. In the case that $\xi \in [-C\varepsilon ^{\alpha } ,C\varepsilon ^{\alpha }]$, we need to wait until $Y^{\varepsilon ,\delta}$ go out of $[-C\varepsilon ^{\alpha '} ,C\varepsilon ^{\alpha '}]$ where $0< \alpha ' <\alpha$. However, until that time, $A^{\varepsilon, \delta}$ almost behave like $Ce^{\mu \tau}$. Indeed, we get
\begin{align}
|A^{\varepsilon, \delta}(\tau, \xi ,x)| &\leq \int _0 ^\tau Y_\xi ^{\varepsilon, \delta} (s,\xi ,x) f''(Y ^{\varepsilon, \delta} (s,\xi ,x) )ds \leq C' \int _0 ^\tau \exp ((\mu +C\varepsilon ^{\alpha '} )s) ds\nonumber \\
& \leq C' \exp (C_1C\varepsilon ^{\alpha '} |\log \varepsilon |) \int _0 ^\tau e^{\mu s} ds \leq \widetilde{C} ( e^{\mu \tau} -1),
\end{align}
for all $\tau \in [0,C_1|\log \varepsilon | \wedge \tau _\varepsilon]$ where the stopping time $\tau _\varepsilon$ is defined by
\begin{align}
\tau _\varepsilon := \inf \{\tau >0 | | Y ^{\varepsilon, \delta} |  \geq C\varepsilon ^{\alpha '} \ for\  some\ \xi \in [-C\varepsilon ^{\alpha } ,C\varepsilon ^{\alpha }] \ or\ x \in D   \}.
\end{align}
To sum up these results, we obtain the next theorem.
\begin{pro}
\label{thm62}
For all $\eta \in (0,1)$, there exist a sequence $\{\delta _\varepsilon \}$ which tends to 0 as $\varepsilon \to 0$ and a positive constant $C_5( \eta )>0$ such that
\begin{align}
\underset {\delta \in (0,\delta_\varepsilon )}{\sup} \underset {\xi \in [-2C_0 ,2C_0]}{\sup} \underset {x \in D}{\sup}|A^{\varepsilon ,\delta} (\tau, \xi ,x)| \leq C_5 (e^{\mu \tau} -1) \ for\ all\  \tau \in [0,C_1|\log \varepsilon | \wedge \tau _4]
\end{align}
holds $P$-a.s.
\end{pro}

\subsection{More detailed estimates}
Next we derive some estimates which we need when we apply the maximal principle.
\begin{lem}
\label{lem64}
There exist $C>0$ and $\alpha >0$ such that
\begin{align}
P\left ( \|Y^{\varepsilon, \delta} _{xx} \|_\infty \leq C\varepsilon ^\alpha \right ) \to 1
\end{align}
as $\varepsilon \to 0$, where $\| \cdot \|_\infty$ is a supreme norm on $\tau \in [0,C_1|\log \varepsilon |]$, $\xi \in [-2C_0,2C_0]$ and $x\in D$.
\end{lem}
\begin{proof}
At first we set a stopping time 
\begin{align}
\tau _7 := \inf \{\tau >0 | \|W_\tau (x) \|_{H^{2+k} (D)} > \varepsilon ^{-\beta}   \}.
\end{align}
for $k >\frac{d}{2}$ and $\beta \in (0, \gamma +\frac{1}{2} -C_f C_1)$. It is easy to check that $P(\tau _7\geq C_1 |\log \varepsilon |)\to 1$ as $\varepsilon \to 0$. Next we derive the estimate for $Y^{\varepsilon, \delta} _{x}$. $Y^{\varepsilon, \delta} _{x}$ satisfies the ODE:
\begin{align}
\begin{cases}
\dot{Y}^{\varepsilon, \delta} _{x} = Y^{\varepsilon, \delta} _{x}f'(Y^{\varepsilon, \delta}) + \varepsilon ^{\gamma +\frac{1}{2}}\partial _x \dot{W}_t ^{(\delta)}(x) ,\\
Y^{\varepsilon, \delta} _{x} (0,\xi,x) = 0 .
\end{cases}
\end{align}
Straightforward computations give us
\begin{align}
|Y^{\varepsilon, \delta} _{x} (\tau ,\xi ,x)| &\leq C_f \int _0 ^\tau |Y^{\varepsilon, \delta} _{x}(s ,\xi ,x)| ds + \varepsilon ^{\gamma +\frac{1}{2} - \beta}
\end{align}
for all $\tau \in [0,C_1 |\log \varepsilon | \wedge \tau _4 \wedge \tau _7]$, $\xi \in [-2C_0 ,2C_0]$ and $x\in D$, because we have that $Y^{\varepsilon, \delta}$ do not go out of some interval until the time $\tau _4$ from Proposition \ref{thm51} and Sobolev's embedding theorem. From Gronwall's inequality and the construction of the process $W_t ^\delta$, we obtain
\begin{align}
\underset{\xi\in[-2C_0,2C_0]}{\sup}\underset{x\in D}{\sup}|Y^{\varepsilon, \delta} _{x} (\tau ,\xi ,x)|\leq C \varepsilon ^{\gamma +\frac{1}{2}-\beta } \exp (C_f \tau)
\end{align}
and this implies that 
\begin{align}
\label{est6-2}
\| Y^{\varepsilon, \delta} _{x} \|_{\infty} \leq C \varepsilon ^{\gamma +\frac{1}{2}-\beta -C_f C_1}.
\end{align}
Now we give the estimate for $Y^{\varepsilon, \delta} _{xx}$ similarly. $Y^{\varepsilon, \delta} _{xx}$ satisfies the ODE:
\begin{align}
\begin{cases}
\dot{Y}^{\varepsilon, \delta} _{xx} = Y^{\varepsilon, \delta} _{xx}f'(Y^{\varepsilon, \delta}) + (Y^{\varepsilon, \delta} _{x})^2 f''(Y^{\varepsilon, \delta}) + \varepsilon ^{\gamma +\frac{1}{2}}\partial _{xx} \dot{W}_t ^{(\delta)} (x) ,\\
Y^{\varepsilon, \delta} _{xx} (0,\xi,x) = 0 .
\end{cases}
\end{align}
The same argument as above, the behavior of $Y^{\varepsilon, \delta}$ and (\ref{est6-2}) conclude this result.
\end{proof}

\begin{lem}
\label{lem65}
There exist $C>0$ and $\kappa >0$ such that for all $C>0$
\begin{align}
P\left ( \left \|\frac{Y^{\varepsilon, \delta} _{\xi x}}{Y^{\varepsilon, \delta} _\xi} \right \|_\infty \leq C\varepsilon ^\kappa \right ) \to 1
\end{align}
as $\varepsilon \to 0$, where $\| \cdot \|_\infty$ is a supreme norm as in Lemma \ref{lem64}.
\end{lem}
\begin{proof}
We note that
\begin{align}
\frac{Y^{\varepsilon, \delta} _{\xi x}}{Y^{\varepsilon, \delta} _\xi} = \frac{\partial}{\partial x} \log Y^{\varepsilon, \delta}_\xi =\int _0 ^\tau Y^{\varepsilon, \delta} _x(s,\xi ,x) f''(Y^{\varepsilon, \delta} (s,\xi ,x))ds
\end{align}
and the estimate (\ref{est6-2}) and the behavior of $Y^{\varepsilon, \delta}$ give us $\frac{Y^{\varepsilon, \delta} _{\xi x}}{Y^{\varepsilon, \delta} _\xi} \leq C\varepsilon ^{\gamma +\frac{1}{2}-C_f C}|\log \varepsilon|$ for all $\tau \in [0,C_1 |\log \varepsilon | \wedge \tau _4 \wedge \tau _7]$, $\xi \in [-2C_0 ,2C_0]$ and $x\in D$.
\end{proof}

\subsection{Proof of the comparison theorem}
Now we prove Proposition \ref{thm61} by using the maximal principle.

\begin{proof}[Proof of Proposition \ref{thm61}]
First of all, we check the initial conditions $\xi$ in $w_\varepsilon ^\pm$ are in $[-2C_0,2C_0]$. When $t \in [0,C_1 \varepsilon | \log \varepsilon |]$ and $\varepsilon$ is sufficiently small, we have that
\begin{align}
u_0 ^+ + \varepsilon C_2 ( e^{\frac{\mu t}{\varepsilon}} -1 ) \leq C_0 + C_2 (\varepsilon ^{1-C_1 \mu} - \varepsilon) \leq 2C_0
\end{align}
because $C_1< \frac{1}{\mu}$. We can show $u_0 ^- - \varepsilon C_2 ( e^{\frac{\mu t}{\varepsilon}} -1 )\geq -2C_0$ in a similar way. Let $\mathcal{L}$ be an operator which is defined by
\begin{align}
\mathcal{L} (u)(t,x) := \dot{u}(t,x) - \Delta u(t,x) - \frac{1}{\varepsilon}f(u(t,x))- \varepsilon ^{\gamma} \dot{W}_t ^{(\delta )}(x).
\end{align}
From the maximal principle, if $\mathcal{L} (w_{\varepsilon ,\delta} ^+)\geq 0$ then $w_{\varepsilon ,\delta} ^+\geq u^{\varepsilon ,\delta}$. Now we fix $t\in[0,C_1 \varepsilon |\log \varepsilon | \wedge \varepsilon \tau _4 \wedge \tau _6 \wedge \varepsilon \tau _7]$ derive an estimate for $\mathcal{L}( w_{\varepsilon ,\delta} ^+)(t,x)$. Until the time $\varepsilon \tau _4 \wedge \tau _6 \wedge \varepsilon \tau _7$, we obtain
\begin{align}
\label{est6-1}
\mathcal{L} (w_{\varepsilon ,\delta} ^+)(t,x) &= \frac{1}{\varepsilon} Y_\tau ^{\varepsilon , \delta } + \mu C_2  e^{\frac{\mu t}{\varepsilon}} Y_\xi ^{\varepsilon , \delta } - \Delta u_0 ^+ Y_\xi ^{\varepsilon , \delta } - \{ \nabla u _0 ^+ \} ^2 Y_{\xi \xi } ^{\varepsilon , \delta  }- 2 \nabla u _0 ^+ Y_{\xi x} ^{\varepsilon , \delta  } - Y_{xx} ^{\varepsilon , \delta  }\nonumber \\
&\ \ \ - \frac{1}{\varepsilon}f(Y ^{\varepsilon , \delta  }) - \varepsilon ^{\gamma} \dot{W}_t ^{(\delta )}(x)\nonumber \\
&= \mu C_2  e^{\frac{\mu t}{\varepsilon}} Y_\xi ^{\varepsilon , \delta } - \Delta u_0 ^{+} Y_\xi ^{\varepsilon , \delta }- \{ \nabla u _0 ^+ \} ^2 Y_{\xi \xi} ^{\varepsilon , \delta  } - 2 \nabla u _0 ^+ Y_{\xi x} ^{\varepsilon , \delta  } - Y_{xx} ^{\varepsilon , \delta  }\nonumber \\
&= \left[ \mu C_2  e^{\frac{\mu t}{\varepsilon}} - \Delta u_0 ^+ - \{ \nabla u _0 ^+ \} ^2 A ^{\varepsilon , \delta  } - 2 \nabla u _0 ^+ \frac{Y_{\xi x} ^{\varepsilon , \delta  }}{Y_{\xi} ^{\varepsilon , \delta  }}\right ] Y_{\xi} ^{\varepsilon , \delta  } - Y_{xx} ^{\varepsilon , \delta  }\nonumber \\
&\geq \left[ \left \{ \mu C_2 - C_5 \{ \nabla u _0 ^+ \} ^2 \right \}  e^{\frac{\mu t}{\varepsilon}} - \Delta u_0 ^+ - 2 \nabla u _0 ^+ \frac{Y_{\xi x} ^{\varepsilon , \delta  }}{Y_{\xi} ^{\varepsilon , \delta  }}\right ] Y_{\xi} ^{\varepsilon , \delta  } - Y_{xx} ^{\varepsilon , \delta  } \nonumber \\
&\geq \left[ \mu C_2 - C_5 \{ \nabla u _0 ^+ \} ^2 - \Delta u_0 ^+ - 2 \nabla u _0 ^+ \frac{Y_{\xi x} ^{\varepsilon , \delta  }}{Y_{\xi} ^{\varepsilon , \delta  }}\right ] Y_{\xi} ^{\varepsilon , \delta  } - Y_{xx} ^{\varepsilon , \delta  } \nonumber \\
&\geq \left ( \mu C_2 - C_5 C_0 ^2 - C_0 - 2 C_0 \varepsilon ^\kappa \right ) Y_{\xi} ^{\varepsilon , \delta  } - \varepsilon ^\kappa
\end{align}
for all $x \in D$ and $\delta \in (0,\delta _\varepsilon )$. Note that the stochastic process $Y_{\xi} ^{\varepsilon , \delta  }$ is positive. The definition of $A^{\varepsilon,\delta}$ gives us the third equality. The forth inequality comes from Proposition \ref{thm62}. We can take $C_2$ large so that the right hand side of (\ref{est6-1}) is larger than 0. 
Thus we have proved that
\begin{align}
w_{\varepsilon ,\delta} ^+ \geq u^{\varepsilon ,\delta}\ for\ all\  t \in \left[0, C_1 \varepsilon |\log \varepsilon | \wedge \varepsilon \tau _4 \wedge \tau _6 \wedge  \varepsilon \tau _7 \right],\ x\in D \ and\  \delta \in (0,\delta _\varepsilon ]
\end{align}
holds $P$-a.s. Now we fix $\varepsilon >0$. For each $\omega \in \Omega$ a.s.-$P$, we can show that the functions $w_{\varepsilon ,\delta} ^+(t,x)$ and $u^{\varepsilon ,\delta}(t,x)$ converge uniformly to $w_{\varepsilon} ^+(t,x)$ and $u^{\varepsilon}(t,x)$ as $\delta \to 0$ respectively from the estimates in Proposition \ref{thm51} and Proposition \ref{thm52}. These convergence preserve the estimate between $w_{\varepsilon ,\delta} ^+(t,x)$ and $u^{\varepsilon ,\delta}(t,x)$, and do not contradict the limit as $\varepsilon \to 0$. And thus we conclude that
\begin{align}
w_{\varepsilon} ^+(t,x) \geq u^{\varepsilon}(t,x)\ for\ every\  t \in \left[0, C_1 \varepsilon |\log \varepsilon | \wedge \varepsilon \tau _4 \wedge \tau _6 \wedge \varepsilon \tau _7 \right] ,\ x\in D
\end{align}
holds $P$-a.s. The converse $w_{\varepsilon} ^- \leq u^{\varepsilon}$ can be proved in a similar way.
\end{proof}

\section{Proof of Theorem \ref{thm1}}
\label{sec8}

In this section, we prove the generation of interface in multi-dimension.
\begin{proof}[Proof of Theorem \ref{thm1}]
At first, we take $\tilde{\gamma} _d$ large enough so that lemmas the propositions from Section \ref{sec3} to \ref{sec6} hold for all $\gamma \geq \tilde{\gamma} _d$. Proposition \ref{thm61} implies that
\begin{align}
\label{est7-3-2}
u^\varepsilon (C_1\varepsilon |\log \varepsilon| \wedge \varepsilon \tau _4 \wedge \tau _6 \wedge \varepsilon \tau _7,x) \leq w_\varepsilon ^+ (C_1\varepsilon |\log \varepsilon| \wedge \varepsilon \tau _4 \wedge \tau _6 \wedge \varepsilon \tau _7,x),
\end{align}
for all $C_1$ which satisfies $0 <C_1 <\frac{1}{\mu}$, holds $P$-a.s. From the definition of $w_\varepsilon ^+$ and $\tau _4$ we can estimate as
\begin{align}
\label{est7-2-2}
w_\varepsilon ^+ (C_1\varepsilon |\log \varepsilon| & \wedge \varepsilon \tau _4 \wedge \tau _6 \wedge \varepsilon \tau _7,x) \nonumber \\
&\leq Y^\varepsilon \left( C_1 |\log \varepsilon|\wedge \varepsilon \tau _4 \wedge \tau _6 \wedge \varepsilon \tau _7 , u_0 ^+ (x) + C_2 ( \varepsilon ^{1-C_1 \mu}-\varepsilon ) ,x \right)\nonumber \\
&\leq Y \left( C_1 |\log \varepsilon|\wedge \varepsilon \tau _4 \wedge \tau _6 \wedge \varepsilon \tau _7 , u_0 ^+ (x) + C_2 ( \varepsilon ^{1-C_1 \mu}-\varepsilon ) \right) +\varepsilon ^\kappa .
\end{align}
From Proposition \ref{thm32}, Lemma \ref{thm43}, (\ref{est7-3-2}) and (\ref{est7-2-2}), we get
\begin{align}
P ( C_1\varepsilon & |\log \varepsilon| \leq \varepsilon \tau _4 \wedge \tau _6 \wedge \varepsilon \tau _7 ) \nonumber \\
& \leq P\left ( u^\varepsilon (C_1\varepsilon |\log \varepsilon|,x) \leq 1 + C\varepsilon ^\kappa \right )
\end{align}
and the left hand side converges to 1 as $\varepsilon \to 0$. This is the upper bound of (i). Remind that the estimate
\begin{align}
|u_0 (x) + C_2 ( \varepsilon ^{1-C_1 \mu}-\varepsilon )| \leq 2C_0
\end{align}
holds for all $x \in D$. The proof of the lower bound is similar. Next we check the conditions (ii) and (iii). We only show (ii). From Proposition \ref{thm61} and the definition of $\tau _4$, we obtain
\begin{align}
u^\varepsilon (C_1\varepsilon |\log \varepsilon| & \wedge \varepsilon \tau _4 \wedge \tau _6 \wedge \varepsilon \tau _7,x) \geq w_\varepsilon ^- (C_1\varepsilon |\log \varepsilon| \wedge \varepsilon \tau _4 \wedge \tau _6 \wedge \varepsilon \tau _7 ,x) \nonumber \\
& \geq Y^\varepsilon \left( C_1 |\log \varepsilon|\wedge \varepsilon \tau _4 \wedge \tau _6 \wedge \varepsilon \tau _7 , u_0 ^- (x) - C_2 ( \varepsilon ^{1-C_1 \mu}-\varepsilon ) ,x \right) \nonumber \\
& \geq Y \left( C_1 |\log \varepsilon|\wedge \varepsilon \tau _4 \wedge \tau _6 \wedge \varepsilon \tau _7 , u_0 ^- (x) - C_2 ( \varepsilon ^{1-C_1 \mu}-\varepsilon ) \right) - \varepsilon ^\kappa
\end{align}
Here we need to observe the neighborhood of $\{ x \in D | u_0 (x) =0 \}$. The condition $\xi >\varepsilon^\alpha$ is equivalent to
\begin{align}
\label{est7-1-2}
u_0 ^- (x) \geq C_2 ( \varepsilon ^{1-C_1 \mu}-\varepsilon ) + \varepsilon^\alpha .
\end{align}
We can take $u_0 - u_0 ^-$ sufficiently small. And thus there exist a positive constant $C>0$ and $u_0 (x) \geq C_2 \varepsilon ^{1-C_1 \mu}$ implies (\ref{est7-1-2}) if $1-C_1 \mu \leq 1 \wedge \alpha$. We take the constant $\beta := 1-C_1 \mu$. The proof of (iii) is similar to that of (ii). The convergence $P ( C_1\varepsilon |\log \varepsilon| \leq \varepsilon \tau _4 \wedge \tau _6 \wedge \varepsilon \tau _7 ) \to 1$, as $\varepsilon \to 0$, complete the proof.
\end{proof}

\section{Proofs of Theorems \ref{thm21} and \ref{thm22}}
\label{sec7}
In this section, we prove the generation and the motion of interface in one-dimension. Before proving, we need some preparations and show the analogous lemmas and propositions to that of Section \ref{sec4}, \ref{sec5} and \ref{sec6}. Let $u^\varepsilon$ be the solution of (\ref{eq:spde2}) and $u^{\varepsilon , \delta}$ be its approximation defined in Section \ref{sec5}. We define super and sub solutions as following;
\begin{align}
\label{supsub}
w_{\varepsilon } ^\pm (t,x) = Y^{\varepsilon } \left( \frac{t}{\varepsilon } , u_0 (x) \pm \varepsilon h(x) \left(e^\frac{\mu t}{\varepsilon}-1 \right) ,x \right) ,
\end{align}
where $Y^{\varepsilon}$ is the solution of SDE (\ref{sde}) with $D$ replaced by $[-1,1]$ and $h \in H^3 (\mathbb{R})$. We take the positive function $h\in H^3(\mathbb{R})$ and prove that they are super and sub solutions of (\ref{eq:spde}) in the next proposition. We can extend the comparison argument in previous section to that of one-dimension.
\begin{lem}
\label{thm43-2}
If $2\gamma +1 -3 \kappa -(\alpha +1)\frac{2C_f }{\mu} >0$ holds, then there exists $C>0$ such that
\begin{align}
P\left ( \|Y^\varepsilon - Y \|_\infty \leq C\varepsilon ^\kappa \right ) \to 1
\end{align}
as $\varepsilon \to 0$, where $\| \cdot \|_\infty$ is a supreme norm on $\tau \in [0,\frac{1}{\mu}|\log \varepsilon |]$, $\xi \in [-2C_0,2C_0]\backslash (-\varepsilon ^{\alpha}, \varepsilon ^{\alpha})$ and $x\in [-1,1]$.
\end{lem}

We approximate $Y^\varepsilon$ as $Y^{\varepsilon ,\delta}$ as in (\ref{ode5-1}). Propositions \ref{thm51}, \ref{thm52}, \ref{thm62} and Lemmas \ref{lem64}, \ref{lem65} with $D$ replaced by $\mathbb{R}$ or $[-1,1]$ can be shown in a similar way to that of these results. We define stopping times $\bar{\tau} _i$ ($i=1,\cdots,7$)
\begin{align}
&\bar{\tau} _2 = \inf \{ \tau >0 | |Y^\varepsilon - Y | > \varepsilon ^\kappa \ for \ some \ \xi \in [-2C_0,2C_0]\backslash (-\varepsilon ^{\alpha}, \varepsilon ^{\alpha}),\  x\in \mathbb{R} \} ,\\
&\bar{\tau} _3 = \inf \{ \tau >0 | \| Y ^{\varepsilon ,\delta} (\tau ,\cdot  ,\cdot )\|_{L^\infty ([-2C_0,2C_0] \times \mathbb{R})} >2C_0+ \delta ' \ for\ some\ \delta \in (0,\delta _\varepsilon ] \},\\
&\bar{\tau} _4 = \bar{\tau}_2 \wedge \bar{\tau} _3, \\
&\bar{\tau} _5 =\inf  \{t>0 | |u^{\varepsilon , \delta} (t,x)| > 2C_0\ for \ some \ x\in \mathbb{R} \ or \ \delta \in (0,\delta_\varepsilon] \},\\
&\bar{\tau} _6 = \bar{\tau}_1 \wedge \bar{\tau} _5,\\
&\bar{\tau} _7 = \inf \{\tau >0 | \|W_\tau (x) \|_{H^3 (\mathbb{R})} > \varepsilon ^{-\beta}   \},
\end{align}
by replacing $D$ to $\mathbb{R}$ or $[-1,1]$, where $\delta ' >0$ and $\beta >0$. $\bar{\tau} _1$ is defined in (\ref{bddst}).

\begin{pro}
\label{thm61-2}
Let $u^\varepsilon$ be the solution of (\ref{eq:spde}) and $w_{\varepsilon } ^\pm$ be defined in (\ref{supsub}). For every $0 < C_1 < \frac{1}{\mu}$,
\begin{align}
 w_{\varepsilon} ^- (t,x) \leq u^{\varepsilon} (t,x) \leq w_{\varepsilon } ^+ (t,x)\ for\ every\  t \in [0, C_1 \varepsilon |\log \varepsilon | \wedge \varepsilon \bar{\tau} _4 \wedge \bar{\tau _6} \wedge \varepsilon \bar{\tau} _7],\  x \in \mathbb{R}
\end{align}
holds $P$-a.s.
\end{pro}
\begin{proof}
We approximate super and sub solutions as following;
\begin{align}
w_{\varepsilon ,\delta} ^\pm (t,x) = Y^{\varepsilon ,\delta} \left( \frac{t}{\varepsilon } , u_0 ^{\pm} (x) \pm \varepsilon h(x) \left(e^\frac{\mu t}{\varepsilon}-1 \right) ,x \right) ,
\end{align}
where $Y^{\varepsilon ,\delta}$ is the solution of ODE (\ref{ode5-1}). The main difference from Proposition \ref{thm61} is that the domain $D$ is replaced by $\mathbb{R}$ and the boundary condition. However, we can apply the comparison theorem to $w_{\varepsilon ,\delta} ^\pm$ and $u^{\varepsilon ,\delta}$ because we consider the totally space $\mathbb{R}$. Then we get
\begin{align}
\label{est6-10}
\mathcal{L} (w_{\varepsilon ,\delta} ^+)(t,x) &\geq \bigg[ \mu h - \{ u' _0 +\varepsilon h' (e^{\frac{\mu t}{\varepsilon}} -1) \} ^2 - \{ \Delta u_0 +\varepsilon \Delta h (e^{\frac{\mu t}{\varepsilon}} -1) \}  \nonumber \\
&\ \ \ - 2 \{ u' _0 +\varepsilon h' (e^{\frac{\mu t}{\varepsilon}} -1) \} \frac{Y_{\xi x} ^{\varepsilon , \delta  }}{Y_{\xi} ^{\varepsilon , \delta  }}\bigg] Y_{\xi} ^{\varepsilon , \delta  } - Y_{xx} ^{\varepsilon , \delta  }
\end{align}
for all $x \in D$ and $\delta \in (0,\delta _\varepsilon )$. Lemma \ref{lem64} and Lemma \ref{lem65} allow us to take $h\in H^3(\mathbb{R})$ large so that the right hand side of (\ref{est6-10}) is larger than 0 and $h(x)$ decay like Definition \ref{def21} for all $|x|\geq 1$ because $u_0\pm 1$ is of class $H^3(\mathbb{R})$ out of the interval $[-1,1]$. See the proof of Proposition 2.5 of \cite{kl} for the detailed condition of $h$. The rest of proof is similar to that of Proposition \ref{thm61}.
\end{proof}

\begin{proof}[Proof of Theorem \ref{thm21}]
We take $\tilde{\gamma}$ large enough as in the proof of Theorem \ref{thm1}. We only need to check the condition in Definition \ref{def21}. By applying Proposition \ref{thm61} and replacing $\varepsilon \tau _4 \wedge \tau _6 \wedge \varepsilon \tau _7$ to $\varepsilon \bar{\tau} _4 \wedge \bar{\tau _6} \wedge \varepsilon \bar{\tau} _7$, We can show the conditions (i), (ii) and (iii) in a similar way to the proof of Theorem \ref{thm1}. Thus we check (iv) and (v) of Definition \ref{def21}. We only consider (iv). From the fact that $Y^\varepsilon = Y$ for all $x>1$ and the definition of $h \in H^3(\mathbb{R})$, we immediately see that
\begin{align}
u^\varepsilon (C_1\varepsilon |\log \varepsilon| & \wedge \varepsilon \bar{\tau} _4 \wedge \bar{\tau _6} \wedge \varepsilon \bar{\tau} _7,x) -1 \leq w_\varepsilon ^+ (C_1\varepsilon |\log \varepsilon| \wedge \varepsilon \bar{\tau} _4 \wedge \bar{\tau _6} \wedge \varepsilon \bar{\tau} _7,x) - 1 \nonumber \\
&\leq Y^\varepsilon \left( C_1 |\log \varepsilon| \wedge \varepsilon \bar{\tau} _4 \wedge \bar{\tau _6} \wedge \varepsilon \bar{\tau} _7, u_0 (x) + h(x) ( \varepsilon ^{1-C_1 \mu}-\varepsilon ) ,x \right)-1\nonumber \\
&=Y \left( C_1 |\log \varepsilon| \wedge \varepsilon \bar{\tau} _4 \wedge \bar{\tau _6} \wedge \varepsilon \bar{\tau} _7, u_0 (x) + h(x) ( \varepsilon ^{1-C_1 \mu}-\varepsilon ) \right)-1\nonumber \\
&\leq u_0 (x) + h(x) ( \varepsilon ^{1-C_1 \mu}-\varepsilon )-1 = \varepsilon ^\kappa g_1 (x) + h(x) ( \varepsilon ^{1-C_1 \mu}-\varepsilon ) \nonumber \\
&\leq \varepsilon^\kappa \bar{g} _1(x),
\end{align}
for all $x>1$, holds  $P$-a.s. for some $\bar{g} _1 \in H^1(\mathbb{R})$. The first inequality in the first line comes from Proposition \ref{thm61}. We get 
\begin{align}
u^\varepsilon (C_1\varepsilon |\log \varepsilon| \wedge \varepsilon \bar{\tau} _4 \wedge \bar{\tau _6} \wedge \varepsilon \bar{\tau} _7,x) -1 \geq - \varepsilon^\kappa \bar{g} _1(x),
\end{align}
in a similar way. The convergence $P ( C_1\varepsilon |\log \varepsilon| \leq \varepsilon \bar{\tau} _4 \wedge \bar{\tau _6} \wedge \varepsilon \bar{\tau} _7 ) \to 1$, as $\varepsilon \to 0$, complete the proof of (iv) in Definition \ref{def21}. We can show (v) in a similar way by virtue of Proposition \ref{thm61}.
\end{proof}

\begin{proof}[Proof of Theorem \ref{thm22}]
We set $C_1\varepsilon |\log \varepsilon |$ as an initial time. Now we construct the super and sub solutions again satisfying PDEs:
\begin{align}
\dot{w}_\varepsilon ^{\pm}(t,x) = \Delta w_\varepsilon ^{\pm}(t,x) + \frac{1}{\varepsilon}f(w_\varepsilon ^{\pm}) + \varepsilon ^\gamma a(x) \dot{W} _t
\end{align}
where the initial values are defined by
\begin{align}
w_\varepsilon ^{\pm}(0,x) :=\begin{cases}
m( \varepsilon ^{-\frac{1}{2}}(x-\xi_0 \pm C\varepsilon ^{1-C_1 \mu}) ) \pm \varepsilon ^\kappa, & x\in [-1,1],\\
1+\varepsilon ^\kappa \tilde{g}_1(x), & x \geq 1,\\
-1+\varepsilon ^\kappa \tilde{g}_2(x),& x \leq -1,
\end{cases}
\end{align}
for $\tilde{g}_1$, $\tilde{g}_2 \in H^1(\mathbb{R})$ where $m$ satisfies the following ODE:
\begin{align}
\begin{cases}
\Delta m + f(m) =0,\ m(0)=0,\ m(\pm \infty) =\pm 1, \\
m\text{ is monotonous increasing}.
\end{cases}
\end{align}
From the form of the generated interfaces, we can take $\tilde{g}_1$, $\tilde{g}_2 \in H^1(\mathbb{R})$ so that $w_\varepsilon ^-(0,x)\leq u_0 ^\varepsilon (x) \leq w_\varepsilon ^+(0,x)$ holds. And thus, these estimates, the maximal principle and the smooth approximation as in Section \ref{sec5} allow us to compare the solutions as below.
\begin{align}
\label{est:comp}
w_\varepsilon ^-(t,x)\leq u^\varepsilon (t,x) \leq w_\varepsilon ^+(t,x)\ holds \ for \ all \ t\in[0,\varepsilon ^{-2\gamma -1}T] \ and \ x \in \mathbb{R}\ P{\text -}a.s.
\end{align}
It is easy to see that $\| w_\varepsilon ^\pm(0,\cdot) - m( \varepsilon ^{-\frac{1}{2}}(\cdot -\xi_0 \pm C\varepsilon ^{1-C_1 \mu}) ) \|_{H^1}\leq C\varepsilon ^\kappa$ for $\kappa >1$ from Theorem \ref{thm21}. And hence, we can complete the proof of the theorem in the same way as the proof of Theorem 1.1 of \cite{kl}.
\end{proof}

{\bf Acknowledgements}

The author would like to thank Professor T. Funaki for his tremendous supports and incisive advices. This work was supported by the Program for Leading Graduate Schools, MEXT, Japan and Japan society for the promotion of science, JSPS.

\end{document}